\documentclass[a4paper]{amsart}
\usepackage{fullpage}
\usepackage[english]{babel}
\usepackage[utf8]{inputenc}
\usepackage{amsmath,amsfonts,amsthm}
\usepackage{mathtools}
\usepackage{graphicx}
\usepackage{subcaption}
\usepackage{type1cm}
\usepackage{eso-pic}
\usepackage{color}
\usepackage{pgfplotstable}
\usepackage{pgfplots}
\pgfplotsset{compat=1.5}

\makeatother
\usepackage{enumitem}
\usepackage{hyperref}
\usepackage{placeins}
\makeatletter
\AtBeginDocument{%
  \expandafter\renewcommand\expandafter\subsection\expandafter
    {\expandafter\@fb@secFB\subsection}%
  \newcommand\@fb@secFB{\FloatBarrier
    \gdef\@fb@afterHHook{\@fb@topbarrier \gdef\@fb@afterHHook{}}}%
  \g@addto@macro\@afterheading{\@fb@afterHHook}%
  \gdef\@fb@afterHHook{}%
}
\makeatother
\usepackage{tikz}
\usepackage{float}
\usetikzlibrary{patterns}
\usetikzlibrary{calc}

\newtheorem{theorem}{Theorem}[section]
\newtheorem{lemma}[theorem]{Lemma}
\newtheorem{remark}[theorem]{Remark}
\newtheorem{problem}[theorem]{Problem}

\newtheorem{definition}[theorem]{Definition}

\newcommand{\logLogSlopeTriangleInv}[6]
{
    \pgfplotsextra
    {
        \pgfkeysgetvalue{/pgfplots/xmin}{\xmin}
        \pgfkeysgetvalue{/pgfplots/xmax}{\xmax}
        \pgfkeysgetvalue{/pgfplots/ymin}{\ymin}
        \pgfkeysgetvalue{/pgfplots/ymax}{\ymax}

        \pgfmathsetmacro{\xArel}{#1}
        \pgfmathsetmacro{\yArel}{#3}
        \pgfmathsetmacro{\xBrel}{#1+#2}
        \pgfmathsetmacro{\yBrel}{\yArel}
        \pgfmathsetmacro{\xCrel}{\xArel}

        \pgfmathsetmacro{\lnxB}{\xmin*(1-(#1+#2))+\xmax*(#1+#2)} 
        \pgfmathsetmacro{\lnxA}{\xmin*(1-#1)+\xmax*#1} 
        \pgfmathsetmacro{\lnyA}{\ymin*(1-#3)+\ymax*#3} 
        \pgfmathsetmacro{\lnyC}{\lnyA+#4*(\lnxA-\lnxB)}
        \pgfmathsetmacro{\yCrel}{\lnyC-\ymin)/(\ymax-\ymin)} 

        \coordinate (A) at (rel axis cs:\xArel,\yArel);
        \coordinate (B) at (rel axis cs:\xBrel,\yBrel);
        \coordinate (C) at (rel axis cs:\xCrel,\yCrel);

        \draw[#5]   (A)-- node[pos=0.5,anchor=south] {1}
                    (B)-- 
                    (C)-- node[pos=0.5,anchor=west] {#6}
                    cycle;
    }
}
\newcommand{\logLogSlopeTriangle}[6]
{

    \pgfplotsextra
    {
        \pgfkeysgetvalue{/pgfplots/xmin}{\xmin}
        \pgfkeysgetvalue{/pgfplots/xmax}{\xmax}
        \pgfkeysgetvalue{/pgfplots/ymin}{\ymin}
        \pgfkeysgetvalue{/pgfplots/ymax}{\ymax}

        \pgfmathsetmacro{\xArel}{#1}
        \pgfmathsetmacro{\yArel}{#3}
        \pgfmathsetmacro{\xBrel}{#1+#2}
        \pgfmathsetmacro{\yBrel}{\yArel}
        \pgfmathsetmacro{\xCrel}{\xArel}

        \pgfmathsetmacro{\lnxB}{\xmin*(1-(#1+#2))+\xmax*(#1+#2)} 
        \pgfmathsetmacro{\lnxA}{\xmin*(1-#1)+\xmax*#1} 
        \pgfmathsetmacro{\lnyA}{\ymin*(1-#3)+\ymax*#3} 
        \pgfmathsetmacro{\lnyC}{\lnyA+#4*(\lnxA-\lnxB)}
        \pgfmathsetmacro{\yCrel}{\lnyC-\ymin)/(\ymax-\ymin)} 

        \coordinate (A) at (rel axis cs:\xArel,\yArel);
        \coordinate (B) at (rel axis cs:\xBrel,\yBrel);
        \coordinate (C) at (rel axis cs:\xCrel,\yCrel);

        \draw[#5]   (A)-- node[pos=0.5,anchor=north] {1}
                    (B)-- 
                    (C)-- node[pos=0.5,anchor=east] {#6}
                    cycle;
    }
}
\title{Non-symmetric isogeometric FEM-BEM couplings}

\author{Mehdi Elasmi}
\address{TU Darmstadt, Institute for Accelerator Science and Electromagnetic Fields, Schlossgartenstra\ss{}e 8, 64289 Darmstadt, Germany}
\email{elasmi@gsc.tu-darmstadt.de}

\author{Christoph Erath}
\address{TU Darmstadt, Department of Mathematics, Dolivostra\ss{}e 15, 64293 Darmstadt, Germany}
\email{erath@mathematik.tu-darmstadt.de}

\author{Stefan Kurz}
\address{TU Darmstadt, Institute for Accelerator Science and Electromagnetic Fields, Schlossgartenstra\ss{}e 8, 64289 Darmstadt, Germany}
\email{kurz@gsc.tu-darmstadt.de}

\thanks{M. Elasmi (corresponding author): TU Darmstadt, Germany; elasmi@gsc.tu-darmstadt.de;
The research of this author was supported in parts by the \emph{Excellence Initiative} 
  of the German Federal and State Governments 
  and the \emph{Graduate School CE} within the \emph{Centre for Computational Engineering} at Technische Universit\"at Darmstadt.}
\thanks{C. Erath: TU Darmstadt, Germany; erath@mathematik.tu-darmstadt.de}
\thanks{S. Kurz: TU Darmstadt, Germany; kurz@gsc.tu-darmstadt.de}
\date{\textbf{\today}}
\DeclareMathOperator{\Div}{div}

\DeclareMathOperator{\curl}{curl}
\DeclareMathOperator{\Span}{span}
\DeclareMathOperator{\Diam}{diam}

\newcommand{\Id}{{\operatorname{Id}}}
\newcommand{\opd}{{\,\operatorname{d}}}

\newcommand{\norm}[1]{{\left\lVert #1 \right\rVert}}
\newcommand{\sphp}{{H^1(\Omega)}}
\newcommand{\splp}{{L^2(\Omega)}}
\newcommand{\spht}{{H^{\frac{1}{2}}(\Gamma)}}
\newcommand{\spdht}{{H^{-\frac{1}{2}}(\Gamma)}}

\newcommand{\scalar}[1]{{\left\langle #1 \right\rangle}}
\newcommand{\R}{{\mathbb R}}

\newcommand{\vertiii}[1]{{\left\vert\kern-0.25ex\left\vert\kern-0.25ex\left\vert #1 
    \right\vert\kern-0.25ex\right\vert\kern-0.25ex\right\vert}}
\begin{document}
\maketitle

\begin{abstract}
We present a coupling of the Finite Element and the Boundary Element Method in an isogeometric framework 
to approximate either two-dimensional Laplace interface problems or 
boundary value problems consisting in two disjoint domains. 
We consider the Finite Element Method in the bounded domains to simulate possibly non-linear materials. 
The Boundary Element Method is applied in unbounded or thin domains where the material behavior is linear.
The isogeometric framework allows to combine different design and analysis tools: first, we consider the same type of NURBS 
parameterizations for an exact geometry representation and second, we use the numerical analysis for the Galerkin approximation. 
Moreover, it facilitates to perform $h$- and $p$-refinements. 
For the sake of analysis, we consider the framework of strongly monotone and Lipschitz continuous operators to ensure well-posedness of the coupled system. 
Furthermore, we provide an a~priori error estimate. We additionally show an improved convergence behavior for the errors in functionals 
of the solution that may 
double the rate under certain assumptions. Numerical examples conclude the work which illustrate the theoretical results. \bigskip
\end{abstract}
\section{Introduction and preliminaries}
\label{sec:introduction}
In the last decades, simulation gained more and more importance as a forth pillar of sciences besides theory, experiments, and observations. 
A successful simulation means a good imitation of some phenomena. This allows the analysis, optimization, and predictions to be made ad-hoc 
with a certain reliability, which depends on the application. For example, this can be achieved for problems that are formulated as boundary and initial 
value problems by choosing the right mathematical model, a good representation of the computational domain, and a suitable numerical method. \\
We encounter in this work two types of model problems. First, we consider a Laplacian interface problem in Section \ref{sec:interfaceProb}. 
Its specificity lies in the combination of a possibly non-linear and non-homogeneous problem in a bounded domain
with a linear and homogeneous 
problem in an unbounded domain. This type of model describes a wide class of engineering and physical applications. One example 
for electromagnetic scattering problems and elastostatics can be found in \cite{steinbach2011note}. 
We address the second type of model in Section \ref{sec:twoDomainProblem}. It is a Boundary Value Problem (BVP) with two 
disjoint domains, which are separated by a (thin) gap. 
In the domains we allow non-linear equations. However, the gap is assumed to be filled with a linear material, 
where the simplest form is air. For a visualization we refer to Figure \ref{fig:2domains}. 
This model is particularly used for the simulation of electro-mechanical energy converters. An example is an electric machine discussed in \cite{zeger}. 
In general, the air gap is very thin. The other two domains are modeled separately for a facilitation of a possible rotation of the 
interior part, called rotor in case of an electric machine. This movement is induced by the interaction of electromagnetic fields in the air gap. The computation of 
forces and torques are therefore one central goal in this type of simulations. Formally, this can be achieved by using the so called 
Maxwell Stress Tensor (MST) method, see, e.g., \cite{kurzMotor}. For this, the solution in the air gap as well as its derivatives are needed. 
These aspects have to be kept in mind for a suitable choice of a numerical method. \\
The coupling of the Finite Element Method (FEM) and the Boundary Element Methods (BEM) appears to be an intuitive and straightforward choice for the above described problems. 
Indeed, the FEM is well established and widely used for possibly non-linear problems in bounded domains. 
On the other side the BEM relies on the transfer of the model problem to an integral representation. 
Further steps then lead to a Galerkin discretization problem on
its boundary with certain integral operators. 
In a post-processing step, a solution can be found in every point of the underlying domain. Hence, BEM is suitable to handle problems
with an unbounded domain where we do not have to truncate the domain since the discretization itself is done on the boundary. 
We remark that a truncation would be mandatory if we would apply FEM. 
Since the BEM discretization takes place on a boundary of the domain, it is also very attractive to get a solution in the thin gap described above. 
For a mesh-based method in the whole domain, e.g.,
like FEM, it is very difficult to find a mesh for such a thin gap, where the numerical method remains stable. 
The discretization on the boundary with BEM and the post-processing afterwards avoids this problem.
However, to apply the BEM we need to know the fundamental solution of the underlying problem. Note that this restricts the application of BEM especially for non-linear problems.
Therefore, we apply BEM in this work for two different applications: first in the exterior unbounded domain and second in the thin air gap. 
In both cases we consider the Laplace operator for the BEM part, where the fundamental solution can be given explicitly. 

In the literature, we distinguish several types of FEM-BEM coupling techniques. 
These coupling procedures differ solely in the considered representation of the Boundary Integral Equations (BIE), which are the basis for BEM. 
In order to introduce briefly the considered BIEs, we envisage first the following Laplace equation
\begin{equation} \label{laplace}
-\Delta u = 0\quad \text{in } \Omega^\kappa,\quad\kappa=0,1,
\end{equation}
where $\Omega^0\subseteq\R^2$ is a bounded domain with Lipschitz boundary $\Gamma$, 
and $\Omega^1=\R^2\backslash\overline{\Omega^0}$ is the corresponding unbounded domain. 
Hence, $\Gamma=\overline{\Omega^0}\cap\overline{\Omega^1}$. Note that \eqref{laplace} is an interior problem for $\kappa=0$
and an exterior problem for $\kappa=1$. In the latter case, we additionally assume the radiation condition
$u(x)=C_{\infty}\log|x|+\mathcal{O}(1/|x|)$ for $|x|\to\infty$ with the unknown constant $C_\infty$, see also Remark~\ref{rem:radiation}. 
For some $x \in \Omega^\kappa$, the solution $u(x)$ is given by the representation formula
\begin{equation} \label{repFormula_intro}
u(x) = (-1)^\kappa \left( \int_\Gamma G(x,y) \phi(y) \opd \sigma_y - \int_\Gamma \partial_{\nu(y)} G(x,y) u_{\vert\Gamma}(y) \opd \sigma_y \right),
\end{equation}  
where 
$G(x,y) = -\frac{1}{2\pi}\log{\vert x-y \vert}$ 
denotes the fundamental solution of the Laplace operator, $\nu(y)$ is an outer normal vector on $\Gamma$ 
pointing outward with respect to $\Omega^0$ at $y \in \Gamma$, 
and $\left(u_{\vert\Gamma},\phi := \partial_\nu u_{\vert\Gamma} \right)$ are the unknown or partially unknown Cauchy data. 
Hereby, the notation $u_{\vert\Gamma}$ means the trace of $u$ with respect to $\Gamma$. Note that we omit to write the trace operators  
in this work due to readability. 
Taking the trace of the representation formula yield to the following BIEs, see, e.g., 
\cite[Chapter~7]{steinbach} for more details,
\begin{equation} 
u_{\vert\Gamma} =  (-1)^\kappa \left( \left( \frac{(-1)^\kappa}{2} -  \mathcal{K} \right) u_{\vert\Gamma}  + \mathcal{V} \phi \right). \label{direct1}
\end{equation}
The invoked Boundary Integral Operators (BIO), the single layer operator $\mathcal{V}$ and the double layer operator $\mathcal {K}$, are given for smooth
enough inputs by
\begin{equation}
 \label{eq:operators}
(\mathcal{V}\phi)(x) =  \int_\Gamma G(x,y) \phi(y) \opd \sigma_y , \quad
(\mathcal{K}u_{\vert\Gamma})(x) =  \int_\Gamma \partial_{\nu(y)} G(x,y) u_{\vert\Gamma}(y) \opd \sigma_y ,
\end{equation}
and can be extended continuously to linear and bounded operators such that
\begin{equation*}
\mathcal{V}:\, H^{s-\frac{1}{2}} (\Gamma) \rightarrow H^{s+\frac{1}{2}} (\Gamma) , \quad 
\mathcal{K}:  H^{s+\frac{1}{2}} (\Gamma) \rightarrow H^{s+\frac{1}{2}} (\Gamma)  ,
\end{equation*}
for $s \in [-\frac{1}{2},\frac{1}{2}]$, c.f. \cite[Theorem~1]{costabel}. In particular, the boundary integral operator $\mathcal{V}$ 
is additionally symmetric, and $\spdht$-elliptic, if $\Diam\left(\Omega\right) <1$, 
see, e.g., \cite[Theorem~6.23]{steinbach}. 
The properties of $\mathcal{V}$ induce the norm equivalence
\begin{equation}
\norm{\psi}^2_\mathcal{V} := \scalar{\psi,\, \mathcal{V}\psi}  \simeq \norm{\psi}^2_\spdht. \label{normV}
\end{equation}
In the previous lines, 
the mentioned spaces have to be understood as follows: for $k > 0$, $H^k(\cdot)$ denotes the standard Sobolev space 
equipped with the usual norm $\Vert \cdot \Vert_{H^k(\cdot)}$. Moreover, the space $H^{k-\frac{1}{2}}(\Gamma)$ is 
the trace space of $H^k(\Omega)$, and spaces with negative exponents $H^{-k}(\Gamma)$ are defined as dual spaces of $H^k(\cdot)$ 
using the natural duality pairing $\left\langle \cdot, \cdot \right\rangle_\Gamma$, which is obtained by the extended 
$L^2$-scalar product $\left( \cdot,\cdot\right)_\Gamma$. 
Furthermore, for the unbounded domain $\Omega^1$ we need functions with local behavior and denote them by 
$H_\mathrm{loc}^1(\Omega^1):=\{v:\Omega^1\to\mathbb{R}\big|\, v_{|K}\in H^1(K) \text{ for all }K\subset\overline{\Omega^1}\text{ compact}\}$. 
Finally, we write ${H^1(\Omega)}^\prime$ for the dual space of $H^1(\Omega)$.
We recall that 
\begin{equation}\label{traceIneq}
\scalar{\psi,v_{\vert\Gamma}}_\Gamma \leq \Vert \psi \Vert_\spdht \Vert v_{\vert\Gamma}\Vert_\spht \leq C_\mathrm{tr} \Vert \psi \Vert_\spdht \Vert v \Vert_{H^1(\Omega)},
\end{equation}
holds for all $v \in H^1(\Omega)$ and $ \psi \in \spdht $, where the trace inequality is encoded with the trace constant $C_\mathrm{tr}>0$. 
For some bounded domain $\Omega$, $\left( \cdot,\cdot\right)_\Omega$ denotes equivalently the standard
$L^2$-scalar product in $\Omega$. 

In the following, we describe a variational ansatz to get a weak form of the model problem. As mentioned above, there
are several coupling strategies possible.
If we describe the FEM part by the weak form of the well-known Green's first formula, a coupling  
with the weak form of \eqref{direct1} ($\kappa=1$) leads to the so called Johnson-N\'ed\'elec coupling introduced in \cite{johnson1980coupling}. 
The combined weak form is non-symmetric even though the model problem itself is symmetric. 
Also a Galerkin discretization leads to a non-symmetric
system of linear equations. Therefore, this coupling is also known as non-symmetric coupling, where the unknowns are
the $u$ of the FEM part and the conormal derivate $\phi$ of the BEM part.
To symmetrize this system in case of a symmetric model problem, we first observe that taking the conormal derivative of \eqref{repFormula_intro}
leads to another integral equation with two other integral operators. 
A modification of the Johnson-N\'ed\'elec coupling with this additional integral equation 
renders the coupled problem symmetric. This procedure appeared first in \cite{costabel1988symmetric} 
and is known as Costabel's symmetric coupling. The price of the symmetry is the use of four BIOs, which is 
computationally more expensive. However, there are still only two unknowns involved.
A coupling method with three unknowns, i.e., additionally the trace $u_{|\Gamma}$ of the BEM part is an unknown, 
is called a three field coupling~\cite{Erath:2012-1}.
A coupling procedure with the so called indirect ansatz is also possible and is called Bielak-MacCamy coupling \cite{bielak1983exterior}. 
With this strategy, however, one unknown of the BEM part has no physical meaning. \\
Because of the advantages of the non-symmetric coupling, we consider in this work only this type of coupling. We will introduce
it formally in Section~\ref{sec:interfaceProb}.
For a long time a mathematical analysis for this coupling was only available for smooth boundaries due to the use of a compactness
argument of the double layer operator $\mathcal{K}$; \cite{johnson1980coupling}. In particular, 
Lipschitz boundaries were excluded . 
However, a decade ago Sayas \cite{sayas2009validity} 
provided in fact the first analysis also for Lipschitz boundaries.
This work influenced several variations and improvements, e.g., \cite{steinbach2011note,aurada2013coupling,erath2017} to mention a few but not all.
Hence, the non-symmetric coupling became a more natural choice, especially, if a part of the model problem is non-symmetric or non-linear. 
In this work, we use the results of \cite{aurada2013coupling}, 
where an extension to non-linear interface problems has been addressed and combine this result with the proof shown in \cite{erath2017}. 
Furthermore, we also use \cite{of2014ellipticity}, 
which extended the proofs for the linear interface problem to certain Boundary Value Problems, i.e, to the second type of
problem considered here. 
For the linear interface problem with a general second order problem in the interior domain, 
we also refer to \cite{erath2017} for a rigorous and, to the authors knowledge, sharpest ellipticity estimate. 
Recently, a complete analysis of a parabolic-elliptic interface problem with a full discretization in the sense of a non-symmetric  
FEM-BEM coupling for spatial discretization 
was published in \cite{Erath:2018-1}. Note that such a system arises, for instance, in the modeling of eddy
currents in the magneto-quasi-static regime \cite{MacCamy:1987}.

Now, having described the weak form of the model problem with the proposed FEM and BEM parts, 
we still need to take two major decisions for a successful simulation: 
a suitable discretization technique, i.e., choosing concrete ansatz spaces for the FEM and BEM, and a good representation of the geometry. 
These steps are typically made independently, which complicates meshing and remeshing procedures 
without altering the original geometry. In order to circumvent this, design step and numerical analysis can be combined by considering the same type of basis 
functions. Hence the geometrical modeling is also used to design ansatz functions in the Galerkin discretization schemes for the approximation of the solution. 
Such a method is proposed in \cite{HUGHES20054135,cottrell2009isogeometric}. It is based on using Non-Uniform Rational B-Splines (NURBS) for the 
unification of Computer Aided Design (CAD) and Finite Element Analysis (FEA). This method is called IsoGeometric Analysis (IGA). 
The first isogeometric BEM simulation of collocation type can be found in~\cite{Costas2009,SCOTT2013197}.
Moreover, fast methods for isogeometric BEM have been successfully implemented 
in \cite{Habrecht2010,MARUSSIG2015458,doelz2019isogeometric}, which reduces the known high computational complexity
of such an application due to the dense matrices produced by the BEM. 
This makes the method more attractive even for more realistic and complex applications, see, e.g., \cite{corno2016isogeometric} and \cite{zeger}. 
A rigorous mathematical analysis for isogemetric FEM started in \cite{Bazilevs2006,da2014mathematical} and for isogemetric Galerkin BEM 
in \cite{GantnerGregor2014AiB,FEISCHL2015362,FEISCHL2016141,GantnerGregor2017Oafs,FUHRER2019571}.
For our purpose the results \cite{da2014mathematical,buffa2020multipatch} 
together with \cite{aurada2013coupling,erath2017,of2014ellipticity} play a central role in proving the validity and an
a~priori error estimate of the FEM-BEM coupling in the isogeometric context, which is done in this manuscript
for the first time.  \\
The rest of this paper is organized as follows: 
in Section \ref{sec:interfaceProb}, the non-linear interface problem is addressed. 
We consider the framework of Lipschitz continuous and strongly monotone operators such as given in \cite{zeidler} and used in \cite{aurada2013coupling}.
Strong monotonicity of the non-symmetric weak form is showed equivalently to \cite{erath2017} by adapting the setting to non-linear operators. 
Moreover, well-posedness of the coupling is stated. Section \ref{sec:iga} is devoted to the Galerkin discretization of the non-symmetric coupling. 
Thereby, we introduce the isogeometric framework and the necessary discrete spaces.
We derive some error estimates for the conforming isogeometric discretization. In Section \ref{sec:twoDomainProblem}, 
we extend the model to a Boundary Value Problem. More precisely, the model domain is split in two disjoint domains, which are separated by a thin (air) gap. 
First, a variational formulation of the coupled problem is derived. Then we show well-posedness and stability of the method.
Furthermore, we discuss a super-convergence result for the evaluation of the solution in the BEM domain. 
In the last Section \ref{sec:results}, we confirm the theoretical results by conducting one numerical example for each model problem. 
The work is completed by some conclusions and an outlook.

\section{Interface problem} \label{sec:interfaceProb}
Let $\Omega \subset \mathbb{R}^2$ be a bounded domain with Lipschitz boundary $\Gamma = \partial \Omega$
and $\Omega^\mathrm{e} := \R^2 \backslash \overline{\Omega}$ the corresponding unbounded (exterior) domain.
Furthermore, to guarantee the $\spdht$-ellipticity of the boundary integral operator 
$\mathcal{V}$, we assume $\Diam\left(\Omega\right) <1$. This assumption can merely be achieved by scaling. We consider the following interface problem:
Find $(u,u^\mathrm{e}) \in H^1(\Omega) \times H_\mathrm{loc}^1(\Omega^\mathrm{e})$ such that
\begin{subequations}\label{classProblem}
\begin{align} 
- \Div \left( \mathcal{U} \nabla u \right) &= f &&\text{in }\Omega,\label{intLaplace}\\
- \Delta u^\mathrm{e} &= 0  &&\text{in }\Omega^\mathrm{e},\label{extLaplace} \\
{ u}_{\vert\Gamma} - {u^\mathrm{e}}_{\vert\Gamma} &= u_0&&\text{on } \Gamma,\label{dirJump} \\ 
\mathcal{U}{\nabla u}_{\vert\Gamma} \cdot \nu - {\nabla u^\mathrm{e}}_{\vert\Gamma} \cdot \nu &= \phi_0 &&\text{on } \Gamma,\label{neuJump} \\ 
u^\mathrm{e} &= \mathcal{O}\left( \vert x \vert^{-1} \right) && \text{for } \vert x \vert \rightarrow \infty.\label{radCondition}
\end{align}
\end{subequations}
We remind that $\nu$ denotes the outer normal vector with respect 
to $\Omega$ and $\mathcal{U}: \, \R^2 \rightarrow \R^2$ is a possibly non-linear diffusion tensor. The right-hand side is given 
by $f \in {H^1(\Omega)}^\prime$, $u_0 \in \spht$ is 
the jump in the Dirichlet data, and $\phi_0 \in \spdht$ the jump in the Neumann data. \\
To ensure the right radiation condition \eqref{radCondition} at infinity, we have to assume the additional condition
\begin{equation*}
\scalar{{\nabla u^\mathrm{e}}_{\vert\Gamma} \cdot \nu,1}_\Gamma = 0.
\end{equation*}
This can be transformed into a compatibility condition on the data, i.e.,
\begin{equation*}
\left(f,1\right)_\Omega + \scalar{\phi_0,1}_\Gamma = 0. 
\end{equation*}
\begin{remark}
 \label{rem:radiation}
Note that the assumption to ensure the radiation condition is only needed in the two dimensional case.
Alternatively, \eqref{radCondition} can be replaced by a logarithmic decay of the solution in two dimensions
to avoid the additional assumption on ${\nabla u^\mathrm{e}}_{\vert\Gamma}$, i.e.,
\begin{equation*}
u^\mathrm{e} = C \log\vert x \vert + \mathcal{O}\left( \vert x \vert^{-1} \right), \quad \text{ for } \vert x \vert \rightarrow \infty,
\end{equation*}
with $C:= \frac{1}{2\pi}\scalar{{\nabla u^\mathrm{e}}_{\vert\Gamma} \cdot \nu,1}_\Gamma$ or equivalently
$C := -\frac{1}{2\pi}\left( \left(f,1\right)_\Omega + \scalar{\phi_0,1}_\Gamma \right)$, which can be easily verified in the weak
formulation below.
\end{remark}
As mentioned above, the diffusion tensor $\mathcal{U}$ can be a non-linear operator. 
To apply standard theory for non-linear operators, see, e.g., \cite{zeidler},
we assume throughout the manuscript that $\mathcal{U}$ is Lipschitz continuous and strongly monotone: 
\begin{enumerate}[label=(A\arabic*)]
\item \label{A1} Lipschitz continuity: 
\[ \exists \, C^\mathcal{U}_\mathrm{Lip}>0  \text{ such that } \vert \mathcal{U}x - \mathcal{U}y \vert \leq C^\mathcal{U}_\mathrm{Lip}\, \vert x - y\vert, \, \forall x,y \in \R^2.\]
\item \label{A2} strong monotonicity: \[ \exists \, C^\mathcal{U}_\mathrm{ell}>0 \text{ such that }  
\left( \mathcal{U} \nabla u - \mathcal{U} \nabla v, \nabla u - \nabla v \right)_\Omega \geq C^\mathcal{U}_\mathrm{ell}\, \norm{\nabla u - \nabla v}^2_\splp, 
\,\forall u,v \in \sphp .\]
\end{enumerate}
The derivation of a non-symmetric variational form follows a standard procedure: 
In the variational form of~\eqref{intLaplace}, we replace the Neumann data by the jump 
condition~\eqref{neuJump} to couple the interior problem with the conormal derivative  
with $\phi := \partial_\nu u^\mathrm{e}_{\vert\Gamma}=\nabla u^\mathrm{e}_{\vert\Gamma}\cdot\nu$ of the exterior problem. 
For the second equation we use the exterior integral equation~\eqref{direct1} with $\kappa=1$,
and insert the jump condition \eqref{dirJump} 
to couple this with the interior trace.\\
Hence, the weak formulation of the non-symmetric coupling problem reads:
Find $\mathbf{u} = (u, \phi) \in H^1(\Omega) \times \spdht$ such that
\begin{align*} 
\left( \mathcal{U}\nabla u, \nabla v \right)_\Omega - \left\langle \phi, v_{\vert\Gamma} \right\rangle_\Gamma &= \left( f, v \right)_\Omega + \left\langle \phi_0, v_{\vert\Gamma} \right\rangle_\Gamma, \\
\left\langle \psi, \left( \frac{1}{2} - \mathcal{K}  \right) u_{\vert\Gamma} \right\rangle_\Gamma + \left\langle \psi, \mathcal{V} \phi \right\rangle_\Gamma &= \left\langle \psi \left( \frac{1}{2} - \mathcal{K}  \right) u_0 \right\rangle_\Gamma
\end{align*}
holds $\forall \mathbf{v} = \left( v, \psi \right) \in H^1(\Omega) \times \spdht$.\\
This variational form can be written in a compact form.
For this we introduce a product space with corresponding norm, i.e.,
\begin{equation}
 \label{eq:spaceH}
 \mathcal{H}:= H^1(\Omega) \times \spdht,\qquad
 \norm{\mathbf{v}}_{\mathcal{H}}:=\big(\norm{v}_{H^1(\Omega)}^2+\norm{\psi}_{\spdht}^2\big)^{\frac{1}{2}} \text{ for } \mathbf{v}=(v,\psi)\in\mathcal{H}.
\end{equation}

\begin{problem} \label{problemJN}
Find $\mathbf{u} \in \mathcal{H}:= H^1(\Omega) \times \spdht$ such that $a(\mathbf{u},\mathbf{v}) = \ell(\mathbf{v})$ holds $\forall \mathbf{v} \in \mathcal{H}$
with the linear form (linear in the second argument) $a:\mathcal{H}\times\mathcal{H}\to \mathbb{R}$,
\begin{equation}
a(\mathbf{u},\mathbf{v}) :=  \left(\mathcal{U}\nabla u, \nabla v\right)_\Omega - \scalar{\phi,v_{\vert\Gamma}}_\Gamma + \scalar{\psi,
\left( \frac{1}{2}-\mathcal{K}\right)u_{\vert\Gamma}}_\Gamma + \scalar{\psi,\mathcal{V}\phi}_\Gamma,\label{linearformproblem}
\end{equation}
and the linear functional $\ell$ on $\mathcal{H}$,
\begin{equation}
\ell(\mathbf{v}) := \left( f, v \right)_\Omega + \left\langle \phi_0, v_{\vert\Gamma} \right\rangle_\Gamma + 
\left\langle \psi, \left( \frac{1}{2} - \mathcal{K}  \right) u_0 \right\rangle_\Gamma.\label{functionalproblem}
\end{equation}
\end{problem}
It is easy to check that $a(\mathbf{v},\mathbf{v})$ is not elliptic, e.g., insert $\mathbf{v} = (1,0)$. 
Hence, \cite{aurada2013coupling} suggested an implicit stabilization where the stabilized problem is equivalent to the original one, i.e., a solution of the original
problem is also a solution of the stabilized one and vice versa.
Thus, the analysis is done with the aid of the stabilized form, i.e., well-posedness is inherited to the original problem.
For implementation purposes, we still use the original problem. 
The stabilized problem reads:
\begin{problem}\label{stabProblem}
Find $\mathbf{u} \in \mathcal{H}$ such that $\widetilde{a}(\mathbf{u},\mathbf{v}) = \widetilde{\ell}(\mathbf{v})$ holds $\forall \mathbf{v} \in \mathcal{H}$,
where we define with
\begin{equation*}
s(\mathbf{v}) := \scalar{1,\left( \frac{1}{2}-\mathcal{K}\right)v_{\vert\Gamma}}_\Gamma + \scalar{1,\mathcal{V}\psi}_\Gamma, \qquad \mathbf{v}=(v,\psi)
\end{equation*}
the stabilized linear form
\begin{equation*}
\widetilde{a}(\mathbf{u},\mathbf{v}):= a(\mathbf{u},\mathbf{v}) + s(\mathbf{u})s(\mathbf{v}),\label{stabBilinearForm}
\end{equation*}
and the functional
\begin{equation*}
\widetilde{\ell}(\mathbf{v}) := \ell(\mathbf{v}) + \scalar{1,\left( \frac{1}{2}-\mathcal{K}\right)u_0}_\Gamma s(\mathbf{v}).\label{stabfunctional}
\end{equation*}
\end{problem}
\begin{lemma}[\cite{aurada2013coupling}]\label{equivalenceFormulation}
The original and the stabilized formulation are equivalent, i.e., $\mathbf{u} \in \mathcal{H}$ solves Problem \ref{problemJN} 
if and only if it solves Problem \ref{stabProblem}, and vice versa.
\end{lemma}
In order to state well-posedness for Problem \ref{stabProblem}, and thanks to Lemma \ref{equivalenceFormulation}
also for Problem \ref{problemJN}, we follow standard results for monotone operatos~\cite{zeidler}.
First, we note that
the form $\widetilde{a}(\mathbf{u},\mathbf{v})$ induces a non-linear operator $\widetilde{\mathcal{A}}: \mathcal{H} \rightarrow \mathcal{H}^\prime$ by 
\begin{equation}\label{inducedOperator}
\scalar{\widetilde{\mathcal{A}} (\mathbf{u}),\mathbf{v}} := \widetilde{a}(\mathbf{u},\mathbf{v}),\quad \forall \mathbf{u},\mathbf{v} \in \mathcal{H},
\end{equation}
where $\mathcal{H}^\prime$ denotes the dual space of $\mathcal{H}$. 
This allows us to prove the following lemma.
\begin{lemma}[\cite{aurada2013coupling,erath2017}] \label{lipschitzmonotone}
Let us consider the non-linear operator $\widetilde{\mathcal{A}}: \mathcal{H} \rightarrow \mathcal{H}^\prime$ defined in \eqref{inducedOperator}
with $\mathcal{H}= H^1(\Omega) \times \spdht$.
The following assertions hold. 
\begin{itemize}
 \item $\widetilde{\mathcal{A}}$ is Lipschitz continuous, i.e., 
there exists $C_\mathrm{Lip} > 0$ such that
\begin{equation*}
 \norm{\widetilde{\mathcal{A}} (\mathbf{u})-\widetilde{\mathcal{A}} (\mathbf{v})}_{\mathcal{H}^\prime} \leq C_\mathrm{Lip} \norm{\mathbf{u}-\mathbf{v}}_\mathcal{H},
 \end{equation*}
for all $\mathbf{u}$, $\mathbf{v} \in \mathcal{H}$.
\item if $C^\mathcal{U}_\mathrm{ell}> \frac{1}{4}$,  there holds that
\begin{equation}\label{stability2.5}
\scalar{\widetilde{\mathcal{A}} (\mathbf{u})-\widetilde{\mathcal{A}} (\mathbf{v}),\mathbf{u}-\mathbf{v}} \geq C_\mathrm{stab} \left( \norm{\nabla u - \nabla v}_\splp^2 + \norm{\phi - \psi}^2_\mathcal{V}+ s(\mathbf{\mathbf{u} - \mathbf{v}})^2 \right),
\end{equation}
for all $\mathbf{u}=(u,\phi)\in \mathcal{H}$, $\mathbf{v}=(v,\psi) \in \mathcal{H}$ 
with the norm $\norm{\psi}^2_\mathcal{V} := \scalar{\psi,\, \mathcal{V}\psi}$ and
with \[C_\mathrm{stab} = \min\left\lbrace 1, \frac{1}{2}\left( 1 +  C^\mathcal{U}_\mathrm{ell} - \sqrt{ \left(C^\mathcal{U}_\mathrm{ell}-1\right)^2 + 1 }\right)\right\rbrace.
\]
\item if $C^\mathcal{U}_\mathrm{ell}> \frac{1}{4}$, then $\widetilde{\mathcal{A}}$ is strongly monotone , 
i.e., there exists $C_\mathrm{ell} > 0$ such that
\begin{equation*}
 \scalar{\widetilde{\mathcal{A}} (\mathbf{u})-\widetilde{\mathcal{A}} (\mathbf{v}),\mathbf{u}-\mathbf{v}} \geq C_\mathrm{ell} \norm{\mathbf{u}-\mathbf{v}}^2_\mathcal{H}
 \end{equation*} 
for all $\mathbf{u}$, $\mathbf{v} \in \mathcal{H}$.
\end{itemize}
\end{lemma}
\begin{proof}
The Lipschitz continuity of $\widetilde{\mathcal{A}}$ follows from the Lipschitz continuity of $\mathcal{U}$, and the continuity of the integral operators.\\
The proof of the second assertion 
follows the lines of \cite[Theorem 1]{erath2017} for $\beta = 1$. We replace  
the coercivity estimate of the bilinear form $\left(\mathcal{U}\nabla u, \nabla v\right)_\Omega$ considered in \cite{erath2017} for a linear $\mathcal{U}$
by the strong monotonicity property of $\mathcal{U}$, i.e, 
\begin{equation*}
\left( \mathcal{U} \nabla u - \mathcal{U} \nabla v, \nabla u - \nabla v  \right)_\Omega \geq C^\mathcal{U}_\mathrm{ell}\, \norm{\nabla u -\nabla v }^2_\splp.
\end{equation*}  
The restriction of $C^\mathcal{U}_\mathrm{ell}$ is a direct result of the use of a contractivity result for the double layer operator $\mathcal{K}$ \cite[Lemma 2.1]{of2013one}
with a constant $C_{\mathcal{K}}\in [\frac{1}{2},1)$, where we use the worst case of $C_{\mathcal{K}}=1$ in the statement.\\
For the last assertion we note the norm equivalence \eqref{normV}, and by a
a Rellich compactness argument it can be shown \cite[Lemma~10]{aurada2013coupling} that
\begin{equation*} 
\vertiii{\mathbf{v}}^2 := \norm{\nabla v}_\splp^2 + \norm{\psi}^2_\mathcal{V}+ s(\mathbf{v})^2 \label{normEq}
\end{equation*}
defines an equivalent norm in $\mathcal{H}$ for some $\mathbf{v} := \left( v, \psi \right) \in \mathcal{H}$. This leads together with \eqref{stability2.5} to the last assertion. 
\end{proof}

The following theorem follows directly from the theoretical result \cite[Theorem~25.B]{zeidler}.
\begin{theorem}[Well-posedness, \cite{zeidler}]\label{dataDep}
Let $C^\mathcal{U}_\mathrm{ell}> \frac{1}{4}$. 
Since the induced operator $\widetilde{\mathcal{A}}$ of $\widetilde a(\cdot,\cdot)$ is strongly monotone and Lipschitz continuous,
there exists 
a unique solution $\mathbf{u}:= (u,\phi) \in \mathcal{H}$  of the variational Problem \ref{stabProblem} 
for any $(f,u_0,\phi_0) \in {H^1(\Omega)}^\prime \times \spht \times \spdht$. Thanks to Lemma~\ref{equivalenceFormulation},
this is also the unique solution of Problem \ref{problemJN}.
\end{theorem}
For some engineering applications, where the non-linear operator $\mathcal{U}$ has a special form, we can state the following stabilization result.
\begin{lemma}\label{dataDep_interface}
  Let all the assumptions of Theorem \ref{dataDep} hold and let
  the non-linear operator $\mathcal{U}$ be of the form 
   $\mathcal{U}\nabla u:=g(|\nabla u|)\nabla u$ with a non-linear function $g:\mathbb{R}\to\mathbb{R}$.
   Then, for the solution $\mathbf{u}:= (u,\phi) \in \mathcal{H}$ of Problem \ref{problemJN}, we have the stability result
 \begin{equation*}
 \norm{\mathbf{u}}_\mathcal{H} \leq C \left( \norm{f}_{\sphp^\prime} + \norm{u_0}_\spht + \norm{\phi_0}_\spdht \right), \quad C > 0.
 \end{equation*}
\end{lemma}

\begin{proof}
Let $\mathbf{v} \in \mathcal{H}$ be arbitrary. 
We know from the strong monotonicity of $\widetilde{\mathcal{A}}$ that
\begin{equation*}
C_\mathrm{ell} \norm{\mathbf{u}-\mathbf{v}}^2_\mathcal{H} \leq \scalar{\widetilde{\mathcal{A}} 
(\mathbf{u})-\widetilde{\mathcal{A}} (\mathbf{v}),\mathbf{u}-\mathbf{v}}.
\end{equation*}
Without loss of generality, 
we choose $\mathbf{v} = \left(0,0\right)$ and note that $\mathcal{U}\nabla v=0$.  
Thanks to Lemma~\ref{equivalenceFormulation} $\mathbf{u}:= \left(u,\phi \right)$ is also the unique solution of the Problem~\ref{stabProblem}.
Thus, we conclude that 
\begin{align*}
C_\mathrm{ell} \norm{\mathbf{u}}^2_\mathcal{H} &\leq \scalar{\widetilde{\mathcal{A}} (\mathbf{u}),\mathbf{u}} = \widetilde{\ell}(\mathbf{u}), \\
&= \left( f, u \right)_\Omega + \left\langle \phi_0, u_{\vert\Gamma} \right\rangle_\Gamma + \left\langle \phi, 
\left( \frac{1}{2} - \mathcal{K}  \right) u_0 \right\rangle_\Gamma + \scalar{1,\left( \frac{1}{2}-\mathcal{K}\right)u_0}_\Gamma s(\mathbf{u}),
\end{align*}
with $s(\mathbf{u}) := \scalar{1,\left( \frac{1}{2}-\mathcal{K}\right)u_{\vert\Gamma}}_\Gamma + \scalar{1,\mathcal{V}\phi}_\Gamma$. 
Next we use inequality \eqref{traceIneq} along with the boundedness of $\mathcal{K}$ and $\mathcal{V}$. Then, rearranging the terms yield to
\begin{align*}
C_\mathrm{ell} \norm{\mathbf{u}}^2_\mathcal{H} &\leq \left( \norm{f}_{\sphp^\prime} + C_\mathrm{tr} \norm{\phi_0}_\spdht + \left( \frac{1}{2}+C^\mathcal{K} \right)^2 C_\mathrm{tr} \norm{u_0}_\spht \right) \norm{u}_\sphp \\
\quad & + \left( \left( \frac{1}{2}+C^\mathcal{K} \right) \left(1 + C^\mathcal{V} \right) \norm{u_0}_\spht \right) \norm{\phi}_\spdht, 
\end{align*}
where $C^\mathcal{K},\,C^\mathcal{V} > 0$ denote the continuity constants of the boundary integral operators $\mathcal{K}$ and $\mathcal{V}$, respectively. 
From this follows the  assertion with a constant $C>0$ that depends on $C^\mathcal{K},\,C^\mathcal{V}, \,C_\mathrm{tr}$, $C_\mathrm{ell}$ and $\Gamma$. 
\end{proof}

\begin{remark}
Because of  
Lemma \ref{equivalenceFormulation}, the results obtained for the stabilized formulation also hold true for 
the original non-symmetric coupling of Problem \ref{problemJN}. 
Hence, in the next section, we only discretize the original problem using a Galerkin approximation. In fact, the stabilized version is only used for analysis purposes.
\end{remark}
\section{Galerkin discretization}\label{sec:iga}
Let $V_\ell \subset \sphp$ and $X_\ell \subset \spdht$ be some finite dimensional subspaces, where the index $\ell$ 
expresses a refinement level, e.g., in a sequence of mesh refinements. We assume that:
\begin{enumerate}[label=(A$3$)]
\item The discrete space $ X_\ell$ contains the constants, i.e., \[ \exists \xi \in \bigcap_{\ell \in \mathbb{N}_0} X_\ell \text{ such that } \scalar{\xi,1}_\Gamma \neq 0. \label{A3} \]
\end{enumerate}
We consider a conforming Galerkin discretization of the Problem \ref{problemJN}. 
Replacing the spaces $\sphp$ and $\spdht$ with $V_\ell$ and $X_\ell$, respectively, yields to the following discrete problem:
Find $\mathbf{u}_\ell = (u_\ell, \phi_\ell) \in \mathcal{H}_\ell=V_\ell \times X_\ell$ such that
\begin{align*} 
\left\langle \mathcal{U}\nabla u_\ell, \nabla v_\ell \right\rangle_\Omega - \left\langle \phi_\ell, v_\ell \right\rangle_\Gamma &= \left\langle f, v_\ell \right\rangle_\Omega + \left\langle \phi_0, v_\ell \right\rangle_\Gamma, \\
\left\langle \psi_\ell, \left( \frac{1}{2} - \mathcal{K}  \right) u_\ell \right\rangle_\Gamma + \left\langle \psi_\ell, \mathcal{V} \phi_\ell \right\rangle_\Gamma &= \left\langle \psi_\ell, \left( \frac{1}{2} - \mathcal{K}  \right) u_0 \right\rangle_\Gamma
\end{align*}
holds $\forall \mathbf{v}_\ell = \left( v_\ell, \psi_\ell \right) \in \mathcal{H}_\ell$.\\
The compact form in the product space $\mathcal{H}_\ell$ reads:
\begin{problem} \label{discreteProblemJN}
Find $\mathbf{u}_\ell=(u_\ell,\phi_\ell) \in \mathcal{H}_\ell=V_\ell \times X_\ell$ such 
that $a(\mathbf{u}_\ell,\mathbf{v}_\ell) = \ell(\mathbf{v}_\ell)$ holds $\forall \mathbf{v_\ell}=(v_\ell,\psi_\ell) \in \mathcal{H}_\ell$.
The linear form $a(\cdot,\cdot)$ and the linear functional $\ell$ are defined in \eqref{linearformproblem} and \eqref{functionalproblem}, respectively.
\end{problem}
Provided that Assumption \ref{A3} is satisfied, the analysis for Problem \ref{discreteProblemJN} is done analogously to the continuous Problem \ref{problemJN} 
since the discrete spaces are conform. In other words, all the above results including the introduction of a stabilized form and
Lemma \ref{equivalenceFormulation} also hold for the subspaces. 
In particular, due to Theorem \ref{dataDep} the discrete solution
$\mathbf{u}_\ell=(u_\ell,\phi_\ell) \in \mathcal{H}_\ell=V_\ell \times X_\ell$ of Problem \ref{discreteProblemJN} exists and is unique.
The following quasi-optimality result in the sense of the C\'ea-type Lemma is a standard but central result, which 
will be needed in Subsection \ref{subsec:estimates} for the a~priori error estimate of the non-symmetric coupling.
\begin{theorem}[Quasi-optimality]\label{cea}
Let Assumption \ref{A3} hold, and $C^\mathcal{U}_\mathrm{ell}> \frac{1}{4}$. 
Moreover, let $\mathbf{u} := \left( u,\phi \right) \in \mathcal{H}$ be the unique solution of Problem \ref{problemJN}, 
and $\mathbf{u}_\ell:= \left( u_\ell,\phi_\ell \right) \in \mathcal{H}_\ell$ 
the solution of its discrete counterpart Problem \ref{discreteProblemJN}. It holds that
\begin{equation*}
\norm{u-u_\ell}_\sphp + \norm{\phi- \phi_\ell}_\spdht \leq C_\text{C\'ea} \min_{v_\ell \in V_\ell, \psi_\ell \in X_\ell} \left(  \norm{u-v_\ell}_\sphp + \norm{\phi- \psi_\ell}_\spdht\right), 
\end{equation*}
with $C_\text{C\'ea} = \frac{C_\mathrm{Lip}}{C_\mathrm{ell}}$.
\end{theorem}
\begin{proof}
The assertion follows as a result of the main theorem on strongly monotone operators \cite[Corollary~25.7]{zeidler}.
That means with $\mathbf{v}_\ell=(v_\ell,\psi_\ell)$, thanks to Lemma \ref{equivalenceFormulation},
the strong monotonicity, Galerkin orthogonality, Cauchy-Schwarz inequality, and the Lipschitz continuity we get
\begin{align*}
C_\mathrm{ell} \norm{\mathbf{u}-\mathbf{u}_\ell}^2_\mathcal{H}
&\leq \scalar{\widetilde{\mathcal{A}} (\mathbf{u})-\widetilde{\mathcal{A}} (\mathbf{u}_\ell),\mathbf{u}-\mathbf{u}_\ell}
=\scalar{\widetilde{\mathcal{A}} (\mathbf{u})-\widetilde{\mathcal{A}} (\mathbf{u}_\ell),\mathbf{u}-\mathbf{v}_\ell}
\leq \norm{\widetilde{\mathcal{A}} (\mathbf{u})-\widetilde{\mathcal{A}} (\mathbf{u}_\ell)}_{\mathcal{H}^\prime}
\norm{\mathbf{u}- \mathbf{v}_\ell}_{\mathcal{H}} \\
&\leq C_\mathrm{Lip} \norm{\mathbf{u}-\mathbf{u}_\ell}_\mathcal{H}\norm{\mathbf{u}-\mathbf{v}_\ell}_\mathcal{H},
\end{align*}
where the assertion follows directly.
\end{proof}
\subsection{Isogeometric Analysis}
The basis functions that are considered for the geometry design 
in the isogeometric framework are used as ansatz functions for the Galerkin discretization.
These functions are typically B-Splines or some extensions of B-Splines, e.g., NURBS, T-Splines etc. \\
In the following, we introduce briefly the concept of Isogeometric Analysis and refer to \cite{cottrell2009isogeometric} for a more detailed introduction, and to \cite{da2014mathematical} and \cite{buffa2020multipatch} for a mathematical analysis of IGA in the FEM and BEM context, respectively.
\begin{definition}
Let $p \in \mathbb{N}$ denote the degree, and $k \in \mathbb{N}$ the number of the B-Spline basis functions, with $k > p$. A knot vector $\Xi := \{ \xi_0, \dots, \xi_{k+p} \}$ is called $p$-open if
\begin{equation*}
0 = \xi_0 = \dots = \xi_{p} < \xi_{p+1} \leq \dots \leq \xi_{k-1} < \xi_{k} = \dots = \xi_{k+p} = 1.
\end{equation*}
Associated to the knot vector $\Xi$, $k$ B-Spline basis functions can be defined recursively for $p \geq 1$ by
\begin{equation*}
b_{i}^{p} (x) = \frac{x - \xi_i}{\xi_{i+p} - \xi_i} b_{i}^{p-1}(x) + \frac{\xi_{i+p+1}-x}{\xi_{i+p+1}-\xi_{i+1}} b_{i+1}^{p-1}(x),
\end{equation*}
for all $i=0 \dots k-1$, starting with piecewise constant basis functions for $p = 0$, namely,
\begin{equation*}
b_{i}^{0}(x) = \begin{cases} 1 & \text{if }  \xi_i \leq x \leq \xi_{i+1} \\ 0 & \text{otherwise}  \end{cases}.
\end{equation*}
Moreover, we denote by $\mathbb{S}_{p}(\Xi) = \Span{\{(b_{i}^{p})_{i=0 \dots k-1}\}}$ the space of B-Splines of degree $p$ and dimension $k$ in the parameter domain over the knot vector $\Xi$. 
\end{definition} 
\begin{definition} Let $f: [0,1] \rightarrow \gamma \subset \R^d$ be a B-Spline mapping defined as 
 \begin{equation*}
f(x) = \sum_{i=0}^{k-1} c_i b_{i}^{p} (x)
\end{equation*} 
with $c_i \in \R^d $ representing an element of a set of control points. The mapping $f$ describes a one-dimensional curve 
embedded in a $d$-dimensional Euclidian space and is called a B-Spline curve. Moreover, we call $\gamma$ a patch if the mapping $f$ is regular. 
\end{definition}
As long as the B-Spline mapping $f$ is regular, we can define B-Spline spaces in the physical domain, i.e., over a patch $\gamma$ by using the following transformation
\begin{equation*}
\iota(f)(u) = u \circ f,
\end{equation*} 
namely, 
\begin{equation*}
\mathbb{S}_{p}(\gamma) = \{v : v = \iota(f)(u)^{-1}, u \in \mathbb{S}_{p}(\Xi) \}.
\end{equation*}  
\begin{definition} B-Spline spaces on higher dimensional domains are constructed by using tensor product relationships. For example, in $2D$, we write
 with $p_1,p_2\in\mathbb{N}$ and $k_1,k_2\in\mathbb{N}$
\begin{equation*}
\mathbb{S}_{p_1,p_2}(\Xi_1,\Xi_2) = \Span{\{(b_{i_1}^{p_1})_{i_1=0 \dots k_1-1} \cdot (b_{i_2}^{p_2})_{i_2=0 \dots k_2-1}\}},
\end{equation*} 
where $p_1$ and $p_2$ denote the degrees in each parametric direction, and $k_1k_2$ is the number of the B-Splines basis functions.
A B-Spline surface is thus represented by $\mathbf{f}(x): [0,1]^2 \rightarrow \omega \subset \R^d$ with
\begin{equation*}
\mathbf{f}(x) = \sum_{i_1=0}^{k_1-1} \sum_{i_2=0}^{k_2-1} \mathbf{c}_{i_1,i_2} \cdot b_{i_1}^{p_1} (x) \cdot  b_{i_2}^{p_2} (x),
\end{equation*} 
and $\mathbf{c}_{i_1,i_2} \in \R^d $ representing an element of a set of control points. If $f$ is regular, we call $\omega$ a patch.\\
Equivalently, the two-dimensional B-Spline space in the physical domain is defined over a patch $\omega$ by 
\begin{equation*}
\mathbb{S}_{p_1,p_2}(\omega) = \{v : v = \iota(\mathbf{f})(u)^{-1}, u \in \mathbb{S}_{p_1,p_2}(\Xi_1,\Xi_2) \}.
\end{equation*} 
\end{definition}
Note that for the sake of simplicity, if $p_1 = p_2= p$, a B-Spline space of degree $p$ should be understood as a B-Spline space of degree $p$ in each parametric direction.\\

The parametrization of curves and surfaces using B-Spline functions allows an exact representation of a large spectrum of geometries. However, they fail to represent conic sections exactly, which are widely present in the design of various engineering applications. In order to circumvent this, Non-Uniform Rational B-Splines (NURBS) are used instead, see \cite{cottrell2009isogeometric} and \cite{piegl2012nurbs}, for instance.
\begin{definition} 
 Let $p,k,p_1,p_2,k_1,k_2\in\mathbb{N}$ as above.
 NURBS mappings can be considered as weighted B-Spline mappings. They can be defined as follows,
 \begin{equation*}
r(x) := \sum_{i=0}^{k-1} \frac{ c_i w_i b_{i}^{p} (x)}{\sum_{j=0}^{k-1} w_j b_{j}^{p} (x)} \quad \text{ in $1$D},
\end{equation*} 
\begin{equation*}
\mathbf{r}(x) := \sum_{i_1=0}^{k_1-1} \sum_{i_2=0}^{k_2-1} \frac{\mathbf{c}_{i_1,i_2} w_{i_1,i_2} \cdot b_{i_1}^{p_1} (x) \cdot  b_{i_2}^{p_2} (x)}{\sum_{j_1=0}^{k_1-1} \sum_{j_2=0}^{k_2-1} w_{j_1,j_2} \cdot b_{j_1}^{p_1} (x) \cdot  b_{j_2}^{p_2} (x)} \quad \text{ in $2$D.}
\end{equation*} 
Thereby, $w_i,\,w_{i_1,i_2} \in \R$ are elements of a vector of dimension $k$ and a matrix of dimension $k_1 \times k_2$, 
containing weighting coefficients of the NURBS, respectively, and $c_i,\mathbf{c}_{i_1,i_2}$ are the control points.
\end{definition}
\begin{remark}
Contrary to B-Splines, NURBS spaces on higher dimensional domains cannot be defined using simple tensor product relationships.     
\end{remark}
In order to guarantee the existence of a regular mapping between the parameter and the physical domain, multiple patches defined through a family of regular parameterizations may in some cases be necessary.
\begin{definition}\label{def:multipatch}
Let $\Omega$ be a two-dimensional Lipschitz domain with boundary $\Gamma$. The domain $\Omega$ is called a multipatch domain, if there exists a family of $N_\Omega$ disjoint patches such that $\Omega = \bigcup_{i} \Omega_i$ and a regular parametrization $\mathbf{r}_i(x): [0,1]^2 \rightarrow \Omega_i$ for every single patch $\Omega_i$, with $0 \leq i < N_\Omega$. Furthermore, we require the parametrization at interfaces to coincide. \\
Equivalently, $\Gamma$ is also considered a multipatch domain with $\Gamma = \bigcup_{i} \Gamma_i$ and $r_i(x): [0,1] \rightarrow \Gamma_i$, $\forall \Gamma_i$, with $0 \leq i < N_\Gamma$.
\end{definition}
Knowing that B-Splines form a partition of unity \cite{piegl2012nurbs}, it is easy to see that B-Splines are a special type of NURBS, 
when the weightings are equal to $1$. 

In the following, if we refer to the geometry, we mean NURBS mappings.  
If we refer to the spaces used for the discretizations, we mean B-Spline mappings. 
The motivation for this follows from \cite{buffa2020multipatch}, namely, the spline 
preserving property of B-Splines is needed for a conforming discretization of the De Rham complex.
\subsection{Error estimates for an isogeometric FEM-BEM discretization}\label{subsec:estimates}
Let the assumptions of Section \ref{sec:interfaceProb} on $\Omega$ hold. We consider the discrete Problem \ref{discreteProblemJN} with $V_\ell = \mathbb{S}^0(\Omega)$ and $X_\ell =\mathbb{S}^2(\Gamma)$, where $\mathbb{S}^0(\Omega)$ and $\mathbb{S}^2(\Gamma)$ are B-Spline spaces defined as in \cite{buffa2020multipatch} and \cite{da2014mathematical}. Namely, 
\begin{equation}\label{S0}
\mathbb{S}^0(\Omega) = \{ u  \in \sphp : u_{|\Omega_i} \in \mathbb{S}_{p,p}(\Omega_i), \, \forall\, 0 \leq i < N_\Omega \},
\end{equation}
and 
\begin{equation}\label{S2}
\mathbb{S}^2(\Gamma) = \{ \phi \in \spdht : \phi_{|\Gamma_i} \in  \mathbb{S}_{p-1}(\Gamma_i), \, \forall \, 0 \leq i < N_\Gamma \}.
\end{equation}
Thereby, $N_\Omega$ and $N_\Gamma$ denote the number of domain patches and boundary patches, respectively. Note that the degrees of the B-Spline spaces \eqref{S0} and \eqref{S2} are solely fixed by one parameter $p>0$.
\begin{definition} Let $\Xi = \{ \xi_0, \dots, \xi_{k+p} \}$ be a p-open knot vector. A patch element in the parameter domain is defined as $[\xi_i, \xi_{i+1}]$, for some $ 0\leq i < k+p $. The local mesh size is defined as the length of an element, i.e., $h_i = \xi_{i+1}-\xi_i$. Furthermore, we denote by $h = \max_{0\leq i < k+p } h_i$ the global mesh size of a single patch. Equivalently, $h$ denotes the largest local mesh size of all patches for a multipatch domain. 
\end{definition}
Throughout the rest of this work, we assume the following:
\begin{enumerate}[label=(A\arabic*)]
 \setcounter{enumi}{3}
\item \label{A4} All knot vectors are $p$-open and locally quasi-uniform, i.e., for all non-empty, neighboring elements $[\xi_{i_1},\xi_{i_1 + 1}]$ and $[\xi_{i_2},\xi_{i_2 + 1}]$, there exists $ \theta \geq 1$, such that \[ \theta^{-1} \leq h_{i_1} h_{i_2}^{-1} \leq \theta. \]
\item \label{A5} The multipatch geometry of $\Omega$ is generated by a family of regular, smooth parameterizations.
\end{enumerate}
\begin{definition}
Let $D = \{\Omega, \Gamma\}$ be a multipatch domain with $n$ patches. For some $s \in \R$, we define the space of patchwise regularity by 
\begin{equation*}
H_\mathrm{pw}^{s}(D) = \{ u \in L^2(D): \norm{u}_{H_\mathrm{pw}^{s}(D)} < \infty \},  
\end{equation*} 
where 
\begin{align}
 \label{eq:pwnorm}
\norm{u}_{H_\mathrm{pw}^{s}(D)}^2 = \sum_{0<i\leq n} \norm{u_{|D_i}}_{H^{s}(D)}^2.
\end{align}
\end{definition}
\begin{lemma}\label{approxProp}
Let $u\in H^1(\Omega) \cap H_\mathrm{pw}^{1+s}(\Omega)$ and $\phi\in H^{-\frac{1}{2}}(\Gamma)\cap H_\mathrm{pw}^{-\frac{1}{2}+s}(\Gamma)$. 
Consider $\mathbb{S}^0(\Omega)$ and $\mathbb{S}^2(\Gamma)$ as given in \eqref{S0} and \eqref{S2}, respectively. There exists $C_0,\,C_2 >0$ such that
\begin{align*}
\inf_{u_\ell \in \mathbb{S}^0(\Omega) }\norm{u-u_\ell}_\sphp  &\leq C_0\, h^{s} \norm{u}_{H_\mathrm{pw}^{1+s}(\Omega)},  \qquad 0\leq s \leq p,\\
\inf_{\phi_\ell \in \mathbb{S}^2(\Gamma)}\norm{\phi-\phi_\ell}_\spdht &\leq C_2\, h^{s} \norm{\phi}_{H_\mathrm{pw}^{-\frac{1}{2}+s}(\Gamma)}, \qquad \frac{1}{2}\leq s \leq p+\frac{1}{2}. 
\end{align*}
\end{lemma}
\begin{proof}
The first estimate is given in \cite[Corollary~2]{buffa2020multipatch}, and the second one follows 
from \cite[Corollary~4]{buffa2020multipatch}.
\end{proof}
\begin{theorem} \label{aPrioriEstimate}
We assume $C^\mathcal{U}_\mathrm{ell}> \frac{1}{4}$.
Let $(u,\phi) \in \mathcal{H}$ be the solution of the Problem \ref{problemJN} and 
let $(u_\ell,\phi_\ell) \in \mathcal{H}_\ell = \mathbb{S}^0(\Omega) \times  \mathbb{S}^2(\Gamma)$ 
be the solution of the discrete Problem \ref{discreteProblemJN}.
Then for $0\leq s\leq \frac{1}{2}$, there holds with $u \in H^1(\Omega) \cap H_\mathrm{pw}^{1+s}(\Omega)$ 
and $\phi \in H^{-\frac{1}{2}}(\Gamma)\cap H_\mathrm{pw}^{0}(\Gamma)$ 
\begin{align*}
\norm{u-u_\ell}_\sphp + \norm{\phi-\phi_\ell}_\spdht \leq C\,h^s \left( \norm{u}_{H_\mathrm{pw}^{1+s}(\Omega)} 
+ \norm{\phi}_{H_\mathrm{pw}^{0}(\Gamma)} \right).
\end{align*}
 
For $\frac{1}{2} \leq s\leq p$, 
and $u \in H^1(\Omega) \cap H_\mathrm{pw}^{1+s}(\Omega)$ 
and $\phi \in H^{-\frac{1}{2}}(\Gamma)\cap H_\mathrm{pw}^{-\frac{1}{2}+s}(\Gamma)$,
we have
\begin{align*}
\norm{u-u_\ell}_\sphp + \norm{\phi-\phi_\ell}_\spdht \leq C\,h^s \left( \norm{u}_{H_\mathrm{pw}^{1+s}(\Omega)} 
+ \norm{\phi}_{H_\mathrm{pw}^{-\frac{1}{2}+s}(\Gamma)} \right)
\end{align*}
with a constant $C>0$.
\end{theorem}
\begin{proof}
From \cite{buffa2020multipatch} we know that $\mathbb{S}^0(\Omega)$ and $\mathbb{S}^2(\Gamma)$ are closed subspaces of $\sphp$ and $\spdht$, respectively. 
Moreover, Assumption \ref{A3} holds true per construction of the B-Spline spaces. Hence, the usual analysis for a conforming Galerkin discretization 
of a non-symmetric FEM-BEM coupling can be considered also in the isogeometric context. Now, 
using Lemma \ref{approxProp} and the quasi-optimality stated in Theorem \ref{cea} yields the assertion. 
\end{proof}
\section{Extension of the model problem} \label{sec:twoDomainProblem}
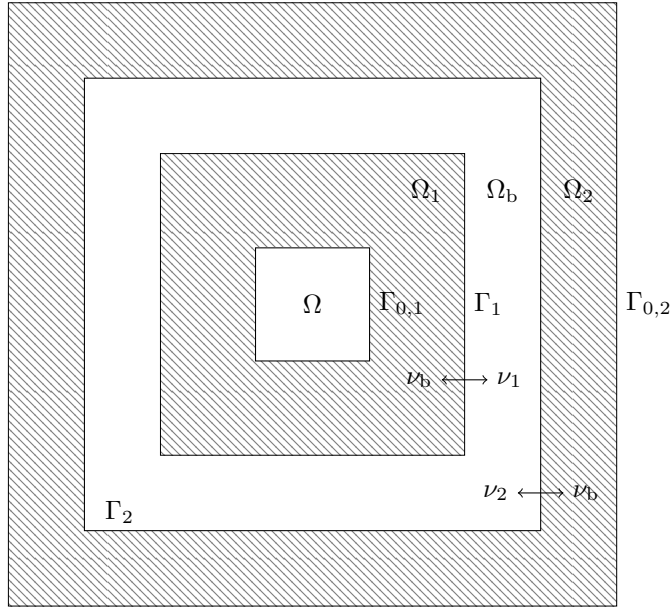
\begin{figure}[!tb]
\begin{center}
\begin{tikzpicture}
\draw[pattern=north west lines, pattern color=gray, very thin] (-2,-2) rectangle ++(8,8);
\draw[fill = white] (-1,-1) rectangle ++(6,6);
\draw[pattern=north west lines, pattern color=gray, very thin] (0,0) rectangle ++(4,4);
\draw[fill = white] (1.25,1.25) rectangle ++(1.5,1.5);
\node at (4.5,3.5) {$\Omega_\mathrm{b}$};
\node at (3.5,3.5) {$\Omega_1$};
\node at (5.5,3.5) {$\Omega_2$};
\node at (2,2) {$\Omega$};
\node[right] at (2.75,2) {$\Gamma_{0,1}$};
\node[right] at (6,2) {$\Gamma_{0,2}$};
\node[right] at (4,2) {$\Gamma_1$};
\node[right] at (-0.85,-0.75) {$\Gamma_2$};

\draw[->] (4,1) --++(0.3,0)node [right] {$\nu_1$} ;
\draw[->] (4,1) --++(-0.3,0)node [left] {$\nu_\mathrm{b}$} ;
\draw[->] (5,-0.5) --++(0.3,0)node [right] {$\nu_\mathrm{b}$} ;
\draw[->] (5,-0.5) --++(-0.3,0)node [left] {$\nu_2$} ;
\end{tikzpicture}
\end{center}
\caption{We see a possible domain arrangement for the boundary value problem discussed in Problem~\eqref{classProbleme} 
with two disjoint domains separated by a gap. The domain $\Omega_\mathrm{b}$ can be very thin and represents, e.g., an air gap. It is surrounded 
by two bounded domains $\Omega_1$ and $\Omega_2$.}\label{fig:2domains}
\end{figure}
Let $\Omega,\Omega_1,\Omega_\mathrm{b},\Omega_2 \subset \mathbb{R}^2$, be bounded Lipschitz domains, 
see Figure \ref{fig:2domains}. We denote by $\Gamma_\mathrm{b} = \Gamma_1 \cup \Gamma_2$ 
the boundary of $ \Omega_\mathrm{b}$ and by $\Gamma_{0,1}$ and $\Gamma_{0,2}$ the Dirichlet boundaries of $\Omega_1$ and  $\Omega_2$, respectively. 
Furthermore, we define 
\begin{equation*}
H^1_0(\Omega_i,\Gamma_{0,i}) := \{ u \in H^1(\Omega_i) : u_{\vert\Gamma_{0,i}}  = 0\}\quad\text{for}\quad i=1,2.
\end{equation*}

We consider the following boundary value problem:
Find $(u_1,u_2,u_\mathrm{b}) \in H^1_0(\Omega_1,\Gamma_{0,1})\times H^1_0(\Omega_2,\Gamma_{0,2}) \times H^1(\Omega_\mathrm{b})$ such that
\begin{subequations}\label{classProbleme}
\begin{align} 
- \Div \left( \mathcal{U}_i \nabla u_i \right) &= f_i &&\text{in }\Omega_i, \quad i=1,2,\label{intLaplacee}\\
- \Delta u_\mathrm{b} &= 0  &&\text{in }\Omega_\mathrm{b} ,\label{extLaplacee} \\
{u_\mathrm{b}}_{\vert\Gamma_i}-{ u_i}_{\vert \Gamma_i} &= u_{0,i}&&\text{on } \Gamma_i,\quad i=1,2,\label{dirJumpe} \\ 
\mathcal{U}_i{\nabla u_i}_{\vert\Gamma_i} \cdot \nu_i + {\nabla u_\mathrm{b}}_{\vert\Gamma_i} \cdot \nu_\mathrm{b} &= \phi_{0,i} &&\text{on } \Gamma_i,\quad i=1,2,\label{neuJumpe} \\ 
{u_i}_{\vert\Gamma_{0,i}} &= 0  &&\text{on } \Gamma_{0,i},\quad i=1,2.  \label{dirConditione}
\end{align}
\end{subequations}
Hereby, $\nu_{i}$ and $\nu_{\text{b}}$ denote the outer normal vector of $\Omega_i$ and $\Omega_\mathrm{b}$, 
respectively, $(f_i,u_{0,i},\phi_{0,i}) \in {H^1(\Omega_i)}^\prime \times H^{\frac{1}{2}}(\Gamma_i) \times H^{-\frac{1}{2}}(\Gamma_i)$ with $i=1,2$
are some given data, and $\mathcal{U}_i$ are possibly non-linear operators with the Assumptions \ref{A1} and \ref{A2}. 
We emphasize that the model problem \eqref{classProbleme} can be used to simulate electric machines, 
see also the example in Section \ref{machine}, which motivates its consideration.
Next, we want to derive a weak formulation for Problem \eqref{classProbleme}.  
We consider the weak form of the two problems in $\Omega_1$ and $\Omega_2$. Hence, we multiply~\eqref{intLaplacee} with
test functions and apply the first Green's identity and get
\begin{equation}\label{3Green}
\left( \mathcal{U}_i\nabla u_i, \nabla v_i \right)_{\Omega_i} 
- \left\langle \mathcal{U}_i\partial_{\nu_i} {u_i}, {v_i}_{\vert\Gamma_i} \right\rangle_{\Gamma_i} = \left( f_i, v_i \right)_{\Omega_i}
\end{equation}
for $i=1,2$. Note that $u_i=0$ on $\Gamma_{0,i}$.
We may transfer \eqref{extLaplacee} in $\Omega_\mathrm{b}$ to an integral equation on $\Gamma_\mathrm{b}$ in order to apply BEM in the following. 
Hence, the (interior) representation formula \eqref{repFormula_intro} ($\kappa=0$)
hold if we replace $u$ by $u_\mathrm{b}$. 
Let $\phi := \partial_{\nu_\mathrm{b}} u_\mathrm{b}$ denote the conormal derivative of $u_\mathrm{b}$ on $\Gamma_\mathrm{b}$,
the BIE is obtained as in Section \ref{sec:introduction}
\begin{equation} \label{interiorProblem}
\mathcal{V} \phi = \left( \frac{1}{2} + \mathcal{K}\right) {u_\mathrm{b}}_{\vert_{\Gamma_\mathrm{b}}}, 
\end{equation} 
where the 
the single layer operator $\mathcal{V}$ and the double layer operator $\mathcal {K}$ are defined in \eqref{eq:operators}
over $\Gamma_\mathrm{b}$ instead of $\Gamma$ but of course with the same fundamental solution $G(x,y)$.
Note that the normal vector $\nu_\mathrm{b}$ points outwards with respect to $\Omega_\mathrm{b}$
since it is considered as an interior problem in our integral equation notation.

In what follows we strongly follow the work of~\cite{of2014ellipticity}, where a boundary value problem with hard inclusion
is considered. 
As in~\cite{of2014ellipticity}, we can derive two equivalent weak formulations. It is enough to consider here only one.
In what follows, the following considerations might help for a better understanding for the weak coupling formulation below.
Note that for a constant it follows $(\frac{1}{2}+\mathcal{K})1 = 0$ on $\Gamma_\mathrm{b}$. Furthermore, if $\mathcal{K}'$
is the adjoint operator of $\mathcal{K}$ and it holds $\mathcal{V}^{-1}\mathcal{K}=\mathcal{K}'\mathcal{V}^{-1}$.
Then, with \eqref{interiorProblem} we see
\begin{align*}
\left\langle \phi, 1 \right\rangle_{\Gamma_\mathrm{b}} =\left\langle \mathcal{V}\phi, \mathcal{V}^{-1}1 \right\rangle_{\Gamma_\mathrm{b}}
=\left\langle ( \frac{1}{2} + \mathcal{K}) {u_\mathrm{b}}_{\vert{\Gamma_\mathrm{b}}}, \mathcal{V}^{-1}1 \right\rangle_{\Gamma_\mathrm{b}}
&=\left\langle  {u_\mathrm{b}}_{\vert{\Gamma_\mathrm{b}}}, ( \frac{1}{2} + \mathcal{K}')\mathcal{V}^{-1}1 \right\rangle_{\Gamma_\mathrm{b}}\\
&=\left\langle  {u_\mathrm{b}}_{\vert{\Gamma_\mathrm{b}}}, \mathcal{V}^{-1}( \frac{1}{2} + \mathcal{K})1 \right\rangle_{\Gamma_\mathrm{b}}=0.
\end{align*}
Note that this $\phi$ together with the representation formula leads to $u_\mathrm{b}$ in $\Omega_\mathrm{b}$, see also \cite[Theorem 7.5]{mclean}.
Therefore, we introduce the following subspace
\begin{equation*}
H_\star^{-\frac{1}{2}}(\Gamma_\mathrm{b}) = \{ \psi \in H^{-\frac{1}{2}}(\Gamma_\mathrm{b}) : \left\langle \psi, 1_{\Gamma_\mathrm{b}} \right\rangle_{\Gamma_\mathrm{b}} = 0\}.
\end{equation*}
Furthermore, similar as in Section \ref{sec:interfaceProb} we introduce a product space with its norm, namely
\begin{align}
 \label{eq:spaceH0}
 \begin{split}
 \mathcal{H}_0&:=H_0^1(\Omega_1,\Gamma_{0,1})\times H_0^1(\Omega_2,\Gamma_{0,2}) \times H_\star^{-\frac{1}{2}}(\Gamma_\mathrm{b}),\\
 \norm{\mathbf{v}}_{\mathcal{H}_0}&:=\big(\norm{v_1}_{H^1(\Omega_1)}^2+\norm{v_2}_{H^1(\Omega_2)}^2+\norm{\psi}_{H^{-\frac{1}{2}}(\Gamma_\mathrm{b})}^2\big)^{\frac{1}{2}}
  \text{ for } \mathbf{v}=(v_1,v_2,\psi)\in\mathcal{H}_0.
  \end{split}
\end{align}

\begin{remark} Instead of considering a subspace and thus eliminating the constants from the solution space, a 
 suitable orthogonal decomposition of $H^{-\frac{1}{2}}(\Gamma_\mathrm{b})$ in the following proofs 
 could also be considered, see \cite{of2014ellipticity}.
\end{remark}
Using $\Gamma_\mathrm{b} = \Gamma_1 \cup \Gamma_2$ and inserting the corresponding jump conditions \eqref{dirJumpe} and \eqref{neuJumpe} in \eqref{interiorProblem} and \eqref{3Green}, 
respectively, yields the following variational problem: \\
Find $\mathbf{u} := (u_1,u_2, \phi) \in \mathcal{H}_0:=H_0^1(\Omega_1,\Gamma_{0,1})\times H_0^1(\Omega_2,\Gamma_{0,2}) \times H_\star^{-\frac{1}{2}}(\Gamma_\mathrm{b})$ such that
\begin{align*} 
\left( \mathcal{U}_1\nabla u_1, \nabla v_1 \right)_{\Omega_1} + \left\langle \phi_{|{\Gamma_1}}, {v_1}_{\vert{\Gamma_1}} \right\rangle_{\Gamma_1} &= \left( f_1, v_1 \right)_{\Omega_1} + \left\langle \phi_{0,1}, {v_1}_{\vert{\Gamma_1}} \right\rangle_{\Gamma_1},\\
\left( \mathcal{U}_2\nabla u_2, \nabla v_2 \right)_{\Omega_2} + \left\langle \phi_{|{\Gamma_2}}, {v_2}_{\vert{\Gamma_2}} \right\rangle_{\Gamma_2} &= \left( f_2, v_2 \right)_{\Omega_2} + \left\langle \phi_{0,2}, {v_2}_{\vert{\Gamma_2}} \right\rangle_{\Gamma_2},\\
\left\langle \psi, \mathcal{V} \phi \right\rangle_{\Gamma_\mathrm{b}} - \sum_{i=1}^2 \left\langle \psi, \left( \frac{1}{2} + \mathcal{K}  \right) {u_i}_{\vert{\Gamma_i}} \right\rangle_{\Gamma_\mathrm{b}} 
&=  \sum_{i=1}^2 \left\langle \psi, \left( \frac{1}{2} + \mathcal{K}  \right) u_{0,i}\right\rangle_{\Gamma_\mathrm{b}} 
\end{align*}
holds $\forall \mathbf{v}:=(v_1,v_2,\psi)\in \mathcal{H}_0$.

As before we first write the problem in a compact form. 

\begin{problem} \label{problemJNBVP}
Find $\mathbf{u}:= (u_1,u_2, \phi) \in \mathcal{H}_0$ such that 
$b(\mathbf{u},\mathbf{v}) = \iota(\mathbf{v})$ holds $\forall \mathbf{v}:= (v_1,v_2,\psi) \in \mathcal{H}_0$.
\end{problem} 
\noindent Thereby,
\begin{equation*}
b(\mathbf{u},\mathbf{v}) := \sum_{i=1}^2\left( \left( \mathcal{U}_i\nabla u_i, \nabla v_i \right)_{\Omega_i} + \left\langle \phi_{|{\Gamma_i}}, {v_i}_{\vert\Gamma_i} \right\rangle_{\Gamma_i} - \left\langle \psi, \left( \frac{1}{2} + \mathcal{K}  \right) {u_i}_{\vert\Gamma_i} \right\rangle_{\Gamma_\mathrm{b}}\right) + \left\langle \psi, \mathcal{V} \phi \right\rangle_{\Gamma_\mathrm{b}},
\end{equation*}
and
\begin{equation*}
\iota(\mathbf{v}) := \sum_{i=1}^2 \left(\left( f_i, v_i \right)_{\Omega_i} + \left\langle \phi_{0,i}, {v_i}_{\vert\Gamma_i} \right\rangle_{\Gamma_i} + \left\langle \psi, \left( \frac{1}{2} + \mathcal{K}  \right) u_{0,i} \right\rangle_{\Gamma_\mathrm{b}}\right).
\end{equation*}
In this case no stabilization is needed, since both subproblems involve a Dirichlet boundary condition. Hence, we prove directly the strong monotonicity of $b(\cdot,\cdot)$. Equivalently to \eqref{inducedOperator}, the form $b(\cdot,\cdot)$ induces a non-linear operator $\mathcal{B}: \mathcal{H}_0 \rightarrow \mathcal{H}_0^\prime$ with 
\begin{equation}\label{inducedOperatorB}
\scalar{\mathcal{B} (\mathbf{u}),\mathbf{v}} := b(\mathbf{u},\mathbf{v})\quad \forall \mathbf{u},\mathbf{v} \in \mathcal{H}_0.
\end{equation}
The next theorem states the strong monotonicity of the method for the extended BVP. It can be considered as an extension to our problem setting of the stability estimate result given in \cite{of2014ellipticity} for an interior Dirichlet BVP of a diffusion equation with a hard inclusion. The key idea therein is to estimate the energy of the bounded finite element domains with the energy of some related problem in the exterior domain. If both corresponding Steklov-Poincar\'e operators are $H^{\frac{1}{2}}(\Gamma)$-elliptic, then it holds with $\lambda>0$ the minimal eigenvalue of the related exterior problem that
\begin{equation}
 \label{eq:lambda}
\lambda \scalar{S^\mathrm{ext}v,v} \leq \scalar{S^\mathrm{int}v,v}, \text{ for all } v \in H^{\frac{1}{2}}(\Gamma),
\end{equation}
where  $S^\mathrm{ext}$ and $S^\mathrm{int}$ are the Steklov-Poincar\'e operators of the exterior and the interior domain, respectively, c.f. \cite{of2014ellipticity}.
\begin{theorem} \label{lipschitzmonotoneBVP}
Let us consider the non-linear operator $\mathcal{B}: \mathcal{H}_0 \rightarrow \mathcal{H}_0^\prime$ defined in \eqref{inducedOperatorB}
with $\mathcal{H}_0= H_0^1(\Omega_1,\Gamma_{0,1})\times H_0^1(\Omega_2,\Gamma_{0,2})\times H_\star^{-\frac{1}{2}}(\Gamma_\mathrm{b})$.
Furthermore, $\lambda_1,\lambda_2>0$ are the eigenvalues of \eqref{eq:lambda} with respect to the domains $\Omega_1$ and $\Omega_2$.
Then the following assertions hold. 
\begin{itemize}
 \item $\mathcal{B}$ is Lipschitz continuous, i.e., 
there exists $C_\mathrm{Lip} > 0$ such that
\begin{equation}
 \norm{\mathcal{B} (\mathbf{u})-\mathcal{B} (\mathbf{v})}_{\mathcal{H}_0^\prime} \leq C_\mathrm{Lip} \norm{\mathbf{u}-\mathbf{v}}_{\mathcal{H}_0}
 \end{equation}
for all $\mathbf{u}$, $\mathbf{v} \in \mathcal{H}_0$.
\item if $C^{\mathcal{U}_i}_\mathrm{ell}> \frac{1}{4\lambda_i}$ for $i = 1,2$, there holds that
\begin{equation}
 \label{eq:stabextended}
\scalar{\mathcal{B} (\mathbf{u})-\mathcal{B} (\mathbf{v}),\mathbf{u}-\mathbf{v}} \geq C_\mathrm{stab} \left( \norm{\nabla u_1 - \nabla v_1}_{L^2(\Omega_1)}^2 + \norm{\nabla u_2 - \nabla v_2}_{L^2(\Omega_2)}^2  + \norm{\phi - \psi}^2_\mathcal{V} \right),
\end{equation}
for all $\mathbf{u}:=(u_1,u_2,\phi)\in \mathcal{H}_0$, $\mathbf{v}:=(v_1,v_2,\psi) \in \mathcal{H}_0$ 
with \[C_\mathrm{stab} = \min\left\lbrace 1, \frac{1}{2}\left( 1 +  C^{\mathcal{U}_1}_\mathrm{ell} - \sqrt{ \left(C^{\mathcal{U}_1}_\mathrm{ell}-1\right)^2 + \frac{1}{\lambda_1} }\right),\frac{1}{2}\left( 1 +  C^{\mathcal{U}_2}_\mathrm{ell} - \sqrt{ \left(C^{\mathcal{U}_2}_\mathrm{ell}-1\right)^2 + \frac{1}{\lambda_2} }\right)\right\rbrace.
\]
\item if  $C^{\mathcal{U}_i}_\mathrm{ell}> \frac{1}{4\lambda_i}$ for $i = 1,2$, then $\mathcal{B}$ is strongly monotone, 
i.e., there exists $C_\mathrm{ell} > 0$ such that
\begin{equation}
 \label{eq:monotoneextended}
 \scalar{\mathcal{B} (\mathbf{u})-\mathcal{B} (\mathbf{v}),\mathbf{u}-\mathbf{v}} \geq C_\mathrm{ell} \norm{\mathbf{u}-\mathbf{v}}^2_{\mathcal{H}_0},
 \end{equation} 
for all $\mathbf{u}$, $\mathbf{v} \in \mathcal{H}_0$.
\end{itemize}
\end{theorem}
\begin{proof}
The Lipschitz continuity follows merely from the Lipschitz continuity of $\mathcal{U}_1$ and $\mathcal{U}_2$ and the continuity of the boundary integral operators. \\
The stability estimate follows strongly the steps of the proofs of \cite[Theorem~2.2.ii.]{of2014ellipticity} and in \cite[Section~5.1]{of2014ellipticity}. Since we are dealing with a different BVP and non-linear material tensors, we sketch the main steps of the proof, for convenience. For ease of notation, let $\mathbf{w} := (w_1,w_2,\xi) = \mathbf{u}-\mathbf{v} = \left( u_1-v_1 ,  u_2-v_2, \phi - \psi \right) \in \mathcal{H}_0$. From \eqref{inducedOperatorB}, we get
\begin{equation} \label{Boperator}
\scalar{\mathcal{B} (\mathbf{u})-\mathcal{B} (\mathbf{v}),\mathbf{w} }:= \sum_{i=1}^2\left( \left( \mathcal{U}_i \nabla u_i - \mathcal{U}_i \nabla v_i, \nabla w_i \right)_{\Omega_i} + \left\langle \xi, \left( \frac{1}{2} - \mathcal{K}  \right) {w_i}_{\vert\Gamma_i} \right\rangle_{\Gamma_\mathrm{b}}\right) + \left\langle \xi, \mathcal{V} \xi \right\rangle_{\Gamma_\mathrm{b}}. 
\end{equation}
First, we start with the domain parts. Provided $\mathcal{U}_i$, $i=1,2$, are strongly monotone, then it holds
\begin{equation*}
\left( \mathcal{U}_i \nabla u_i - \mathcal{U}_i \nabla v_i, \nabla w_i \right)_{\Omega_i} \geq C^{\mathcal{U}_i}_\mathrm{ell}\, \norm{\nabla w_i}^2_{L^2(\Omega_i)}.
\end{equation*} 
For $w_i \in H^1_0(\Omega_i,\Gamma_{0,i}) $, we now consider the splitting $w_i = \overline{w}_i + w_{0,i}$, 
where $\overline{w}_i$ is the harmonic extension of ${w_i}_{|\Gamma_i}$ and $w_{0,i} \in H^1_0(\Omega_i,\partial \Omega_i) $ as in \cite{erath2017}, for instance.
From this follows 
\begin{equation*}
\norm{\nabla w_i}^2_{L^2(\Omega_i)} =  \norm{\nabla w_{0,i}}^2_{L^2(\Omega_i)} +  \scalar{ S_i {w_i}_{|\Gamma_i},{w_i}_{|\Gamma_i} }_{\Gamma_i}, 
\end{equation*}
where $S_i$, $i=1,2$, denote the interior Steklov-Poincar\'e operators of the bounded domains $\Omega_1$ and $\Omega_2$, respectively. Hence, 
\begin{equation}\label{splitting}
\left( \mathcal{U}_i \nabla u_i - \mathcal{U}_i \nabla v_i, \nabla w_i \right)_{\Omega_i} \geq C^{\mathcal{U}_i}_\mathrm{ell} \left(  \norm{\nabla w_{0,i}}^2_{L^2(\Omega_i)} +  \scalar{ S_i {w_i}_{|\Gamma_i},{w_i}_{|\Gamma_i} }_{\Gamma_i}\right).
\end{equation} 
Next, by using the contractivity of $\mathcal{K}$, as given in \cite[Lemma~2.1]{of2014ellipticity} (we consider here the worst case $C_\mathcal{K} =1$), as well as the invertibility of $\mathcal{V}$, we obtain
\begin{equation*}
  \left\langle \xi, \left( \frac{1}{2} - \mathcal{K}  \right) {w_i}_{\vert\Gamma_i} \right\rangle_{\Gamma_\mathrm{b}} \leq \norm{\xi}_\mathcal{V} \sqrt{ \scalar{S_i^\mathrm{ext} {w_i}_{\vert\Gamma_i},{w_i}_{\vert\Gamma_i}}_{\Gamma_\mathrm{b}}}, \quad i = 1,2,
 \end{equation*}
where $S_i^\mathrm{ext}:  H^{\frac{1}{2}}(\Gamma_\mathrm{b}) \rightarrow  H^{-\frac{1}{2}}(\Gamma_\mathrm{b})$ are the Steklov-Poincar\'e operators associated to the corresponding exterior eigenvalue problem, see \cite[Section~2.2]{of2014ellipticity}. Similarly, we assume the following spectral equivalence
\begin{equation*}
\scalar{S_i^\mathrm{ext}{w_i}_{\vert\Gamma_i},{w_i}_{\vert\Gamma_i}}_{\Gamma_i} \leq \frac{1}{\lambda_i } \scalar{S_i {w_i}_{\vert\Gamma_i},{w_i}_{\vert\Gamma_i}}_{\Gamma_i}, \text{ for all } w_i \in H_0^1(\Omega_i,\Gamma_{0,i}),
\end{equation*}
where $\lambda_i$, $i = 1,2$ are characterized as minimal eigenvalues of the related problem. Thus, 
\begin{equation}\label{contractivity}
  \left\langle \xi, \left( \frac{1}{2} - \mathcal{K}  \right) {w_i}_{\vert\Gamma_i} \right\rangle_{\Gamma_\mathrm{b}} \leq \norm{\xi}_\mathcal{V} \sqrt{ \frac{1}{\lambda_i } \scalar{S_i {w_i}_{\vert\Gamma_i},{w_i}_{\vert\Gamma_i}}_{\Gamma_i}}, \quad i = 1,2.
 \end{equation}
 Inserting \eqref{splitting} and \eqref{contractivity} in \eqref{Boperator}, 
 $\left\langle \xi, \mathcal{V} \xi \right\rangle_{\Gamma_\mathrm{b}} = \norm{\xi}_\mathcal{V}^2$
 and some manipulations as in the proof of \cite[Theorem1]{erath2017}
 
 lead to the assertion.\\
 To prove the last claim we consider $\mathbf{v} := \left( v_1,v_2, \psi \right) \in \mathcal{H}_0$.
 Note that $v_1=0$ on $\Gamma_{0,1}$ and $v_2=0$ on $\Gamma_{0,2}$ with $|\Gamma_{0,1}|,|\Gamma_{0,2}|>0$.
 Due to Friedrich's inequality and \eqref{normV} it follows that
 $\norm{\nabla v_1}_{L^2(\Omega_1)}^2+\norm{\nabla v_2}_{L^2(\Omega_2)}^2 + \norm{\psi}^2_\mathcal{V}$
 is an equivalent norm on $\mathcal{H}_0$. Thus, \eqref{eq:monotoneextended} follows directly from \eqref{eq:stabextended}.
\end{proof}%

Equivalently to Theorem \ref{dataDep}, the strong monotonicity and the Lipschitz continuity of the non-linear operator $\mathcal{B}$ 
yields the well-posedness of Problem \ref{problemJNBVP} 
for any $\left(f_i,u_{0,i},\phi_{0,i} \right)  \in {H^1(\Omega_i)}^\prime \times H^{\frac{1}{2}}(\Gamma_i) \times H^{-\frac{1}{2}}(\Gamma_i)$ with $i=1,2$. \\
As for the interface problem, we consider a conforming Galerkin discretization in the sense of an isogeometric FEM-BEM discretization. 
Namely, the discrete problem is obtained by replacing 
$\mathbf{u} := (u_1,u_2,\phi) \in \mathcal{H}_0 := H_0^1(\Omega_1,\Gamma_{0,1})\times H_0^1(\Omega_2,\Gamma_{0,2})\times H_\star^{-\frac{1}{2}}(\Gamma_\mathrm{b})$ 
in Problem \ref{problemJNBVP} with 
$\mathbf{u}_\ell := (u_{1,\ell},u_{2,\ell},\phi_\ell) \in \mathcal{H}_{0,\ell}:=\mathbb{S}^0(\Omega_1,\Gamma_{0,1}) \times \mathbb{S}^0(\Omega_2,\Gamma_{0,2}) \times \mathbb{S}^2(\Gamma_\mathrm{b})$. 
Note that accordingly to the notation in the continuous setting, $\mathbb{S}^0(\Omega,\Gamma)$ 
denotes the B-Spline space $\mathbb{S}^0(\Omega)$ of order $p$ as defined in \eqref{S0} with a Dirichlet boundary $\Gamma \subseteq \partial \Omega$
and $\mathbb{S}^2(\Gamma_\mathrm{b})$ is defined in~\eqref{S2}. 

\begin{problem}
 \label{discreteextended}
 Find $\mathbf{u}_\ell := (u_{1,\ell},u_{2,\ell},\phi_\ell) \in \mathcal{H}_{0,\ell}:=\mathbb{S}^0(\Omega_1,\Gamma_{0,1}) \times \mathbb{S}^0(\Omega_2,\Gamma_{0,2}) \times \mathbb{S}^2(\Gamma_\mathrm{b})$
 such that $b(\mathbf{u}_\ell,\mathbf{v}_\ell) = \iota(\mathbf{v}_\ell)$ holds $\forall \mathbf{v}_\ell:= (v_{1,\ell},v_{2,\ell},\psi_\ell) \in \mathcal{H}_{0,\ell}$.
\end{problem}

Analogously to the interface problem, we state in the following theorem the quasi-optimality in the sense of the C\'ea-type Lemma of 
the Galerkin discretization of Problem \ref{problemJNBVP}, as well as an a~priori error estimate for the introduced B-Spline discretization.
To simplify the presentation, we introduce in this section a piecewise defined product space 
\begin{align}
 \label{eq:Hspw}
 \mathcal{H}_\mathrm{pw}^s:=\big(H^1(\Omega_1)\cap H_\mathrm{pw}^{1+s}(\Omega_1)\big)
 \times \big(H^1(\Omega_2)\cap H_\mathrm{pw}^{1+s}(\Omega_2)\big)
 \times \big(H^{-\frac{1}{2}}_\star(\Gamma_\mathrm{b})\cap H_\mathrm{pw}^{-\frac{1}{2}+s}(\Gamma_\mathrm{b})\big)
\end{align}
for $s\geq 0$, which is used to get convergence rates with the aid of Lemma \ref{approxProp}.
The corresponding norm defined in the sense of \eqref{eq:pwnorm} is denoted by $\Vert \cdot \Vert_{\mathcal{H}_\mathrm{pw}^s}$.
\begin{theorem}\label{quasiOpt&Apriori}
For $i=1,2$, let $C^{\mathcal{U}_i}_\mathrm{ell}> \frac{1}{4\lambda_i}$, where
$\lambda_i>0$ are the eigenvalues of \eqref{eq:lambda} with respect to the domains $\Omega_i$.
 Moreover, 
let $\mathbf{u}\in \mathcal{H}_0$ be the solution of Problem \ref{problemJNBVP} and $\mathbf{u}_\ell \in \mathcal{H}_{0,\ell}$ 
be the discrete solution of Problem~\ref{discreteextended}. Then the following results hold:
\begin{itemize}
\item Quasi-optimality: 
\begin{equation}\label{quasi}
\Vert \mathbf{u} - \mathbf{u}_\ell \Vert_{\mathcal{H}_0} \leq C_\text{C\'ea} \min_{\mathbf{v}_\ell\in \mathcal{H}_{0,\ell}} \Vert \mathbf{u} - \mathbf{v}_\ell \Vert_{\mathcal{H}_0},
\end{equation}
where $C_\text{C\'ea} = \frac{C_\mathrm{Lip}}{C_\mathrm{ell}}$. 
\item A~priori estimate: For $\frac{1}{2}\leq s\leq p$, there holds with $\mathbf{u}\in \mathcal{H}_\mathrm{pw}^s$
\begin{equation*}
\Vert \mathbf{u} - \mathbf{u}_\ell \Vert_{\mathcal{H}_0} \leq C\, h^s \Vert \mathbf{u} \Vert_{\mathcal{H}_\mathrm{pw}^s}
\end{equation*}
with a constant $C>0$.
For $0\leq s\leq \frac{1}{2}$, there holds a result similar to Theorem \ref{aPrioriEstimate} 
with $\phi\in H^{-\frac{1}{2}}_\star(\Gamma_\mathrm{b})\cap H_\mathrm{pw}^{0}(\Gamma_\mathrm{b})$.
\end{itemize}
\end{theorem} 
\begin{proof}
Quasi-optimality follows from the strong monotonicity and Lipschitz continuity stated in Theorem \ref{lipschitzmonotoneBVP}, by following the lines of 
Theorem \ref{cea}. The a~priori estimate follows from the quasi-optimality and Lemma \ref{approxProp}, as is done in Theorem \ref{aPrioriEstimate} for the interface problem.
\end{proof}
The non-linear operators $\mathcal{U}_i$, $i = 1,2$, are now considered to have the form $\mathcal{U}_i\nabla u:=g_i(|\nabla u|)\nabla u$ with non-linear functions $g_i: \R\rightarrow \R$. 
Similarly to the interface problem, we state the following stability result.
\begin{lemma}\label{dataDep_interfaceBVP}
Let $C^{\mathcal{U}_i}_\mathrm{ell}> \frac{1}{4\lambda_i}$, $i=1,2$, with
$\lambda_i$ as in Theorem \ref{quasiOpt&Apriori}. 
Furthermore, the non-linear operators $\mathcal{U}_i$, $i = 1,2$, shall have the form 
$\mathcal{U}_i\nabla u:=g_i(|\nabla u|)\nabla u$ with the non-linear functions $g_i: \R\rightarrow \R$.
Moreover, let $\mathbf{u}\in \mathcal{H}_0$ be the unique solution of Problem \ref{problemJNBVP} and $\left(f_i,u_{0,i},\phi_{0,i} \right)  \in {H^1(\Omega_i)}^\prime \times H^{\frac{1}{2}}(\Gamma_i) \times H^{-\frac{1}{2}}(\Gamma_i)$, with $i=1,2$, be some suitable inputs. There exists $C > 0$ such that 
 \begin{equation*}
 \norm{\mathbf{u}}_{\mathcal{H}_0} \leq C \sum_{i=1}^2 \left( \norm{f_i}_{{H^1(\Omega_i)}^\prime} + \norm{u_{i,0}}_{H^{\frac{1}{2}}(\Gamma_i)} + \norm{\phi_{i,0}}_{H^{-\frac{1}{2}}(\Gamma_i)} \right).
 \end{equation*}
\end{lemma}%
\begin{proof}
We know from the strong monotonicity of $\mathcal{B}$ that
\begin{equation*}
C_\mathrm{ell} \norm{\mathbf{u}-\mathbf{v}}^2_{\mathcal{H}_0} \leq \scalar{\mathcal{B} (\mathbf{u})-\mathcal{B} (\mathbf{v}),\mathbf{u}-\mathbf{v}}
\end{equation*}
holds for all $\mathbf{u}$, $\mathbf{v} \in \mathcal{H}_0$. Without loss of generality, 
we choose $\mathbf{v} = \left(0,0,0\right)$ and note that $\mathcal{U}_i\nabla v_i=0$, $i =1,2$, for our specific non-linearity.  
Since $\mathbf{u}:= \left(u_1,u_2,\phi \right)$ is the unique solution of the problem,
we conclude that 
\begin{align*}
C_\mathrm{ell} \norm{\mathbf{u}}^2_{\mathcal{H}_0} &\leq \scalar{\mathcal{B} (\mathbf{u}),\mathbf{u}} = \iota(\mathbf{u}), \\
&= \sum_{i=1}^2 \left(\left( f_i, u_i \right)_{\Omega_i} + \left\langle \phi_{0,i}, {u_i}_{\vert\Gamma_i} \right\rangle_{\Gamma_i} + \left\langle \phi, \left( \frac{1}{2} + \mathcal{K}  \right) u_{0,i} \right\rangle_{\Gamma_\mathrm{b}}\right) .
\end{align*}
Using inequality \eqref{traceIneq} along with the boundedness of $\mathcal{K}$ and $\mathcal{V}$, and rearranging the terms yields the assertion. 
\end{proof}%

In many practical applications, one is not directly interested in the solution $(u_1,u_2,\phi)$ of Problem \ref{problemJNBVP} rather than 
in some derived quantities. These quantities are, for example, evaluated in the exterior/air gap domain. As it can be observed for standalone BEM applications, 
estimating the error in functionals of the solution may lead to a so called super-convergence, i.e., 
linear functionals of the solution may converge better than the solution in the energy norm, see \cite[Section~4.2.5]{sauter}. 
With enough regularity the convergence rate doubles.  \\ 
In the following, this behavior is also showed for the coupled problem. For this, we use the following Aubin-Nitsche argument, similarly to \cite[Theorem~4.2.14]{sauter}.
\begin{theorem}\label{aubinNitsche}
Let $\mathfrak{F}\in\mathcal{H}_0'$ be a continuous and linear functional on the solution 
$\mathbf{u} := (u_1,u_2,\phi) \in \mathcal{H}_0:= H_0^1(\Omega_1,\Gamma_{0,1})\times H_0^1(\Omega_2,\Gamma_{0,2})\times H_\star^{-\frac{1}{2}}(\Gamma_\mathrm{b})$ 
of Problem \ref{problemJNBVP} and $\mathbf{u}_\ell := (u_{1,\ell},u_{2,\ell},\phi_\ell) \in \mathcal{H}_{0,\ell}$
is the discrete solution of Problem~\ref{discreteextended}. 
Furthermore, let $\mathbf{w} \in \mathcal{H}_0$ be the unique solution of the dual problem
\begin{equation} \label{dualP}
b(\mathbf{v},\mathbf{w}) = \mathfrak{F}(\mathbf{v}),
\end{equation} 
for all $\mathbf{v} \in \mathcal{H}_0$. 
Then there exists a constant $C_1>0$ such that
\begin{align}
 \label{functional1}
 \vert \mathfrak{F}(\mathbf{u}) - \mathfrak{F}(\mathbf{u}_\ell) \vert &
 \leq C_1 \,   \Vert \mathbf{u} - \mathbf{v}_\ell \Vert_{\mathcal{H}_0} \Vert \mathbf{w} - \mathbf{z}_\ell \Vert_{\mathcal{H}_0}
\end{align}
for arbitrary $\mathbf{v}_\ell\in \mathcal{H}_{0,\ell}$, $\mathbf{z}_\ell\in\mathcal{H}_{0,\ell}$.
Furthermore, let $\frac{1}{2} \leq s,t \leq p$ and remember the product space defined in \eqref{eq:Hspw}.
Provided $\mathbf{u}$ and $\mathbf{w}$ are additionally 
in $\mathcal{H}_\mathrm{pw}^s$ and $\mathcal{H}_\mathrm{pw}^t$, respectively, 
there exists a constant $C_2>0$ such that 
\begin{align}
 \label{functional2}
\vert \mathfrak{F}(\mathbf{u}) - \mathfrak{F}(\mathbf{u}_\ell) \vert &
\leq C_2 \,  h^{s+t} \Vert \mathbf{u}\Vert_{\mathcal{H}_\mathrm{pw}^s} \Vert \mathbf{w} \Vert_{\mathcal{H}_\mathrm{pw}^t}.
\end{align}
\end{theorem}
\begin{proof}
The proof follows strongly the lines in \cite[Theorem~4.2.14]{sauter}. Since we allow non-linearities, we give a brief sketch.
First of all, we note that Theorem \ref{lipschitzmonotoneBVP} holds for arbitrary functions. Thus, well-posedness, and hence the existence of a unique solution can be established also for the dual problem~\eqref{dualP}. 
Furthermore, the dual problem \eqref{dualP}, the Galerkin 
orthogonality $b(\mathbf{u}-\mathbf{u}_\ell,\mathbf{z}_\ell)=0$ for all $\mathbf{z}_\ell\in\mathcal{H}_{0,\ell}$, 
and the Lipschitz continuity of the form $b(\cdot,\cdot)$ yield to
\begin{equation*} 
\vert \mathfrak{F}(\mathbf{u}) - \mathfrak{F}(\mathbf{u}_\ell) \vert 
=\vert\mathfrak{F}(\mathbf{u}-\mathbf{u}_\ell)\vert
=\vert b(\mathbf{u}-\mathbf{u}_\ell,\mathbf{w}-\mathbf{z}_\ell)\vert
\leq \widetilde{C} \Vert \mathbf{u} - \mathbf{u}_\ell \Vert_{\mathcal{H}_0} \Vert \mathbf{w} - \mathbf{z}_\ell \Vert_{\mathcal{H}_0},
\end{equation*}
for arbitrary $\mathbf{z}_\ell = (z_{1,\ell},z_{2,\ell},\varphi_\ell) \in \mathcal{H}_{0,\ell}$.

With~\eqref{quasi} we get the claim~\eqref{functional1}.
Since Lemma \ref{approxProp}  holds for arbitrary $\mathbf{u}$, \eqref{functional2} follows from \eqref{functional1}.
\end{proof}

\begin{remark}\label{rem:dependence_pos}
In practice the functional of Theorem \ref{aubinNitsche} may be, e.g., the representation formula of the BEM part $\Omega_{\mathrm{b}}$, i.e.,
for $\mathbf{u}=(u_1,u_2,\phi)\in\mathcal{H}_0$ there holds
\begin{equation*}
 \mathfrak{F}(\mathbf{u}):=\sum_{i=1}^2\Big(\int_{\Gamma_i} G(x,y)\phi_{\vert\Gamma_i}\opd \sigma_y
 - \int_{\Gamma_i}\partial_{\nu(y)}G(x,y)(u_{i\vert\Gamma_i}+u_{0,i\vert\Gamma_i})\opd \sigma_y\Big).
\end{equation*}
Next, let us assume the regularity $\mathbf{u}\in\mathcal{H}_\mathrm{pw}^p$ of the solution \ref{problemJNBVP}  
 and $\mathbf{w}\in\mathcal{H}_\mathrm{pw}^p$
 of its dual problem ~\eqref{dualP}, where the spaces are defined in \eqref{eq:Hspw}.
Then with the discrete solution $\mathbf{u}_\ell \in \mathcal{H}_{0,\ell}$ and \eqref{functional2}
 we calculate the pointwise error in $\Omega_{\mathrm{b}}$ as
\begin{equation}
 \label{eq:superconvergence}
|u_{\mathrm{b}}(x) - u_{\mathrm{b},\ell}(x)|=|\mathfrak{F}(\mathbf{u})(x)-\mathfrak{F}(\mathbf{u_\ell})(x)|\leq C h^{2p},
\end{equation}
which is the maximal possible super-convergence. 
 Since the constant $C$ depends on $\Vert \mathbf{u}\Vert_{\mathcal{H}_\mathrm{pw}^p}$ 
 and $\Vert \mathbf{w} \Vert_{\mathcal{H}_\mathrm{pw}^p}$, a possible estimate of these norms would probably 
 involve their right-hand sides.
 The right-hand side of the dual problem~\eqref{dualP} is the functional $\mathfrak{F}(\mathbf{u})$.
 Thus, the constant $C$ might include a 
 factor like $\sum_{i=1}^2 \left( \norm{G(x,\cdot)}_{H^{\frac{1}{2}+p}(\Gamma_i)} + \norm{\partial_\nu G(x,\cdot)}_{H^{-\frac{1}{2}+p}(\Gamma_i)} \right)$. 
Note that this term is finite for all $ x  \in \R^2\backslash\Gamma_\mathrm{b}$ and $p\geq 0 $. 
However, because of the singularity of the kernels, 
its tends to infinity when approaching the boundaries. 
Thus, also $C$ from \eqref{eq:superconvergence} might tend to infinity.
This effect is even more severe, if we consider functionals that 
involve derivatives of the kernels, e.g., for the computation of forces and torques using the Maxwell Stress Tensor. 
Finally, we mention that the regularity assumptions might only hold for smooth surfaces.
 \end{remark}
\begin{remark}
 \label{rem:evaluation}
 In the linear case the dual problem to Problem \ref{problemJN} or Problem \ref{problemJNBVP} is the corresponding Bielak-MacCamy coupling \cite{aurada2013coupling}.
 Similar results as in Theorem~\ref{aubinNitsche} and in Remark~\ref{rem:dependence_pos} for the extended Problem~\ref{classProbleme}
 can be gained for the interface Problem~\ref{classProblem}.
\end{remark}

\section{Numerical illustration }\label{sec:results}
To illustrate the theoretical results, we consider for each model problem one example. 
The description of NURBS geometric entities are obtained by means of the \textsf{NURBS toolbox} included 
in \textsf{GeoPDEs}, which is implemented in \textsf{MATLAB}, see \cite{geopdes}. In the same spirit, 
the required matrices associated to the boundary integral operators are implemented by using, adapting, 
and supplementing some structures of \textsf{GeoPDEs}. The implementation of the BIOs for arbitrary ansatz 
functions is performed numerically using standard Gauss-Legendre quadrature for regular contributions 
and by means of some Duffy-type transformations with a subsequent combination of logarithmic and Gaussian quadrature 
for the singular parts, see, e.g., \cite[Chapter~4.3]{bantle}. In the following, the $\mathcal{H}$ 
and $\mathcal{H}_0$-norm of~\eqref{eq:spaceH} and~\eqref{eq:spaceH0}, respectively, are computed by Gaussian quadrature.
However, we replace the non-computable norm $\norm{\cdot}_\spdht$
by the equivalent norm $\norm{\cdot}^2_\mathcal{V}$ stated in \eqref{normV}.
  Moreover, we measure the error for the evaluated solution in the BEM-domain in the following way.
First we define an evaluation path $\Gamma_\mathrm{e}$ in the BEM domain. For a certain number $N$ of evaluations 
points $x_i\in\Gamma_\mathrm{e}$, $i=1,\ldots,N$,
 $x_i\not = x_j$,  with $i\not = j$,
 we define the pointwise error as
\begin{equation}
  \label{pointwiseError}
\text{error} = \max_{i=1,\ldots,N} \vert u^\mathrm{e}(x_i) - u^\mathrm{e}_\ell(x_i) \vert \quad\text{and}\quad
\text{error} = \max_{i=1,\ldots,N} \vert u_\mathrm{b}(x_i) - u_{\mathrm{b},\ell}(x_i) \vert.
\end{equation} 
Here, $u^\mathrm{e}_\ell$ and $u_{\mathrm{b},\ell}$
are the discrete evaluations of the corresponding representation formula~\eqref{repFormula_intro} with the Cauchy data from
the corresponding discrete coupling problem. Note that for both problem types the trace has to be calculated with the
aid of the jump condition~\eqref{dirJump} and~\eqref{dirJumpe}, respectively.

In all our experiments, we consider uniform $h$-refinement, 
for different degrees of B-Splines, starting from the minimal degrees needed to represent the geometry exactly. 
Increasing the degree of basis functions is called $p$-refinement. Furthermore, note that the number of elements in every 
$h$-refinement step is calculated by $N_\mathrm{e} = \frac{N_\Omega}{h^{d}}$, where $N_\Omega$ denotes the number of patches 
and $d$ the dimension of the considered manifold. The element size $h$ is obtained in every refinement 
step $\ell \in \mathbb{N}$ by $h = \frac{1}{\ell+1}$.
\subsection{Single domain}\label{subsec:mexicanHat}
In the first example, we consider a square domain $\Omega := \left( -0.25,0.25 \right)^2$ 
and denote its boundary by $\Gamma$. We parametrize $\Omega$ as a single patch domain using linear B-Spline functions 
in each parametric direction. It is obvious that Assumption \ref{A5} about the multipatch geometry is satisfied.  

Moreover, we consider the interface problem \eqref{classProblem} with a linear material tensor $\mathcal{U} := \Id$. 
As in \cite{erath2017}, we prescribe the exact solutions 
\begin{equation*}
u(x) = \left( 1-100 x_1^2 - 100 x_2^2 \right)e^{-50(x_1^2+x_2^2)}, \qquad x=\left( x_1,x_2 \right) \in \Omega, 
\end{equation*}
and 
\begin{equation*}
u^\mathrm{e}(x) = \log(\sqrt{x_1^2 +x_2^2}), \qquad  x \in \Omega^\mathrm{e}.
\end{equation*}
We calculate the jumps $u_0$, $\phi_0$, and the right-hand side $f$ appropriately.
Solving the coupled problem using the isogeometric framework, as described in the previous section, yields a 
discrete solution $(u_\ell,\phi_\ell) \in \mathcal{H}_\ell := \mathbb{S}^0(\Omega) \times  \mathbb{S}^2(\Gamma)$. 
An isogeometric approach for this example is not mandatory since the domain $\Omega$ is standard Cartesian, see, 
e.g., \cite{erath2017}. However, we 
want to demonstrate our higher order coupling approach and in particular the super-convergence behaviour of this example.
Figure \ref{fig:mexicanHat} shows the solution $u_\ell \in \mathbb{S}^0(\Omega)$ in the interior domain, 
as well as the exterior solution $u_\ell^\mathrm{e}$ in a subset of $\Omega^\mathrm{e}$, which we call an evaluation domain
$ \Omega^\mathrm{e}_\mathrm{e} := \left( -\frac{1}{2},\frac{1}{2} \right)\backslash \overline{\Omega}$. 
The exterior solution $u_\ell^e$ is obtained from the representation formula \eqref{repFormula_intro} ($\kappa=1$) with
the computed Cauchy data $({u_\ell}_{|\Gamma}-u_0,\phi)$ from our discrete solution of the interface problem. 
Thereby, the degree of the considered B-Spline space for the domain discretization is $p=2$ and its dimension corresponds to an $h$-refinement level $\ell = 20$.
\begin{figure}[!tb]
\begin{center}
\includegraphics[width=11.5cm]{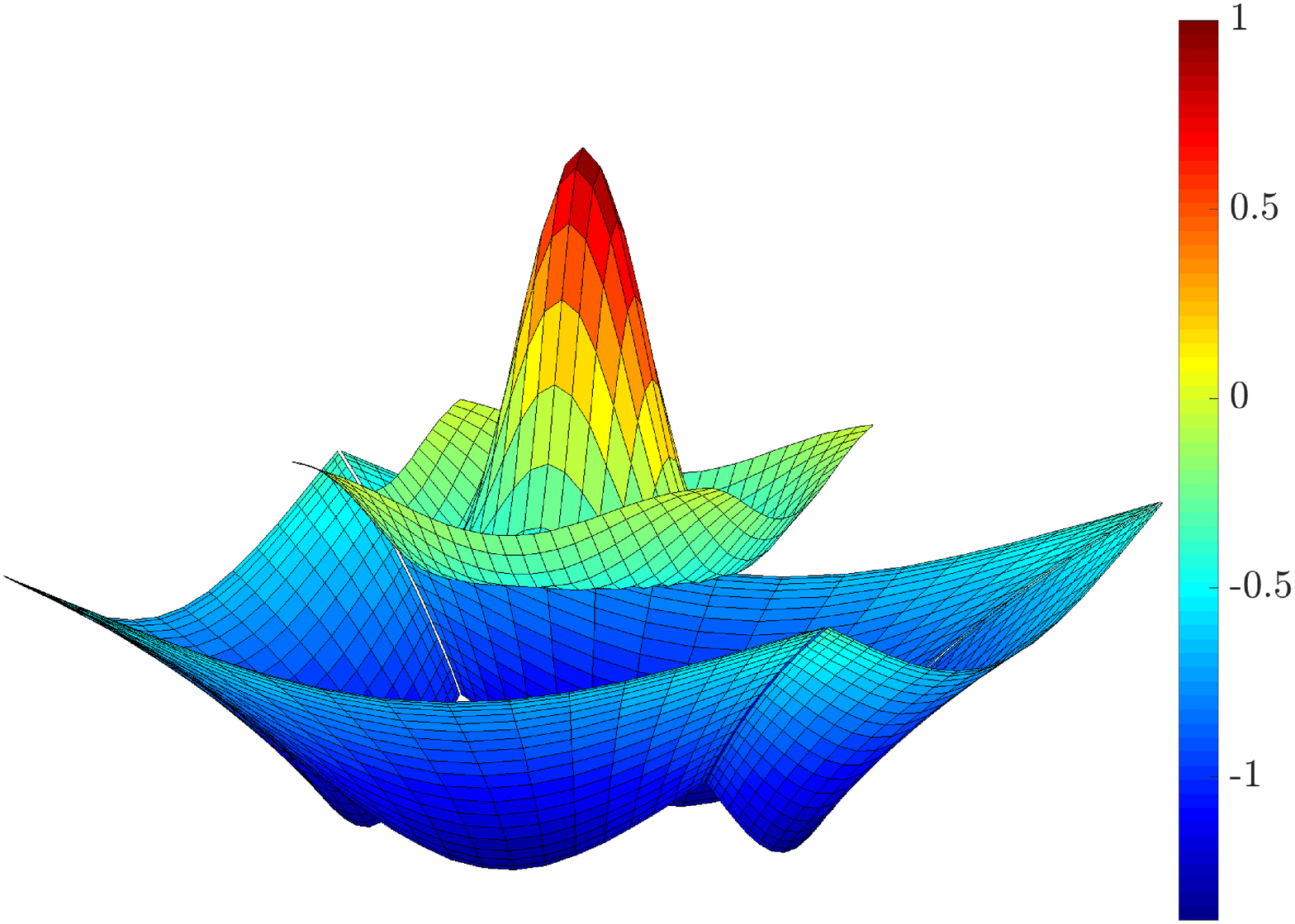}
\end{center}
\caption{Solution $\left(u_\ell,u_\ell^\mathrm{e}\right) \in \mathbb{S}^0(\Omega) \times \mathbb{S}^0(\Omega^\mathrm{e}_\mathrm{e})$ of example in Section \ref{subsec:mexicanHat},
where we restrict the representation of the exterior solution to
$\Omega^\mathrm{e}_\mathrm{e} := \left( -\frac{1}{2},\frac{1}{2} \right)\backslash \overline{\Omega}$. 
The considered B-Spline space corresponds to $\mathbb{S}^0(\Omega)$ 
with degree $p=2$ and an $h$-refinement level $\ell = 20$. The exterior solution is obtained by the evaluation of $25$ points in each of the $4$ exterior patches. } \label{fig:mexicanHat}
\end{figure}%
As a first numerical experiment, we analyze the convergence of the isogeometric FEM-BEM coupling 
with respect to  the norm $\sqrt{\norm{u-u_\ell}^2_{H^1(\Omega)}+\norm{\phi-\phi_\ell}^2_\mathcal{V}}$, 
which is equivalent to $\mathcal{H}$-norm in $\Omega$. Since the solution 
is smooth, the expected order of convergence is equal to the degree of the considered discrete space $\mathcal{H}_\ell$, 
as given in the a~priori estimate from Theorem \ref{aPrioriEstimate}. In Figure \ref{fig:convergence_mh_int} we observe 
the predicted optimal convergence of the method for B-Spline spaces of degree $p = 1,2,3$. \\
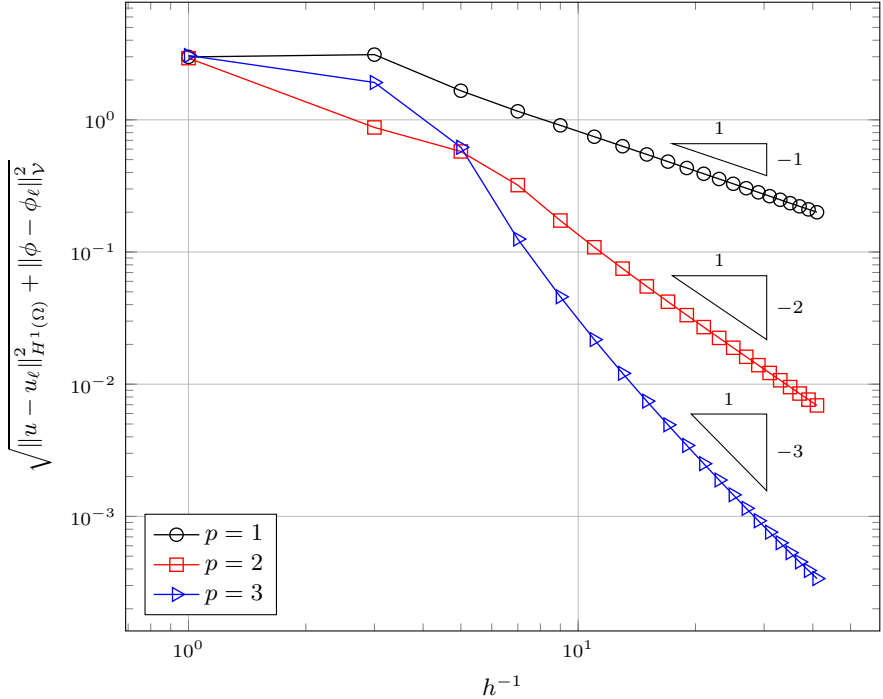
\begin{figure}[!tb]
\centering
\begin{tikzpicture}
\begin{loglogaxis}[width=11.5cm,
xlabel={\small $h^{-1}$},
ylabel={\small $\sqrt{\norm{u-u_\ell}^2_{H^1(\Omega)}+\norm{\phi-\phi_\ell}^2_\mathcal{V}}$ },
 font={\scriptsize},grid=major,
legend entries={$p=1$,$p=2$,$p=3$},
legend style={font=\small,at={(0.025,0.025)},
anchor=south west,thin}]
\addplot [color=black, mark=o, line width=0.5pt, mark size=2.5pt, mark options={solid, black}] table {data/convergence_p1_a__error_H_0to40_mh.dat};
\addplot [color=red, mark=square, line width=0.5pt, mark size=2.5pt, mark options={solid, red}] table {data/convergence_p2_a__error_H_0to40_mh.dat};
\addplot [color=blue, mark=triangle, line width=0.5pt, mark size=3pt, mark options={solid, rotate=270, blue}] table {data/convergence_p3_a__error_H_0to40_mh.dat};
\logLogSlopeTriangleInv{0.85}{-0.1}{0.345}{-3}{black}{$-3$};
\logLogSlopeTriangleInv{0.85}{-0.125}{0.565}{-2}{black}{$-2$};
\logLogSlopeTriangleInv{0.85}{-0.125}{0.775}{-1}{black}{$-1$};
\end{loglogaxis}
\end{tikzpicture}
\caption{Convergence of discrete solution $(u_\ell,\phi_\ell)\in\mathcal{H}_\ell$ to the solution $(u,\phi)\in\mathcal{H}$
for the example in Section \ref{subsec:mexicanHat}.
The considered B-Spline spaces have the degrees $p=1,2,3$, respectively, and the error is presented 
in the norm $\sqrt{\norm{\cdot}^2_{H^1(\Omega)}+\norm{\cdot}^2_\mathcal{V}}$, which is equivalent to the standard  
$\mathcal{H}=H^1(\Omega)\times H^{-\frac{1}{2}}(\Gamma)$ norm.} \label{fig:convergence_mh_int}
\end{figure}~
In the second experiment, we investigate the convergence of the solution in the exterior domain. Note that 
our exterior solution is smooth. 
At a first step, we evaluate the solution on an evaluation path $\Gamma_\mathrm{e}$, which we define here as the boundary 
of $\left( -0.35,0.35 \right)^2$. We calculate the error according to \eqref{pointwiseError} with $N=20$ evaluations
points. In Figure \ref{fig:convergence_Exterior_mh}, 
we observe a doubling of the convergence rates with respect to the pointwise error, which
confirms the theoretical considerations in Remark \ref{rem:dependence_pos}, see also Remark \ref{rem:evaluation}. 
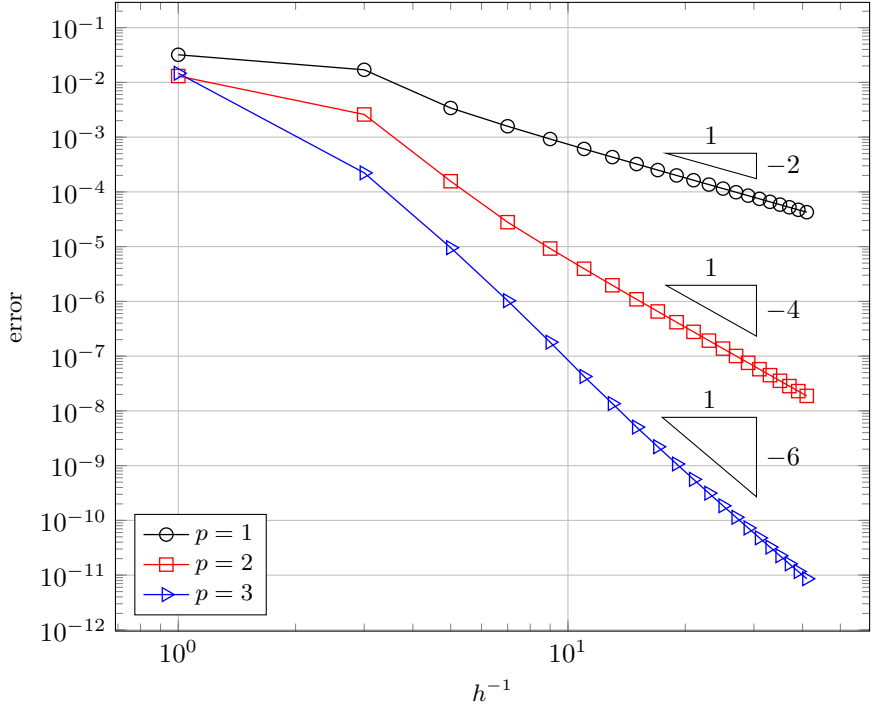
\begin{figure}[!tb]
\centering
\begin{tikzpicture}
\begin{loglogaxis}[width=11.5cm,
xlabel={\small $h^{-1}$},
ylabel={\small error},grid=major,
legend entries={$ p =1 $,$ p = 2 $,$p=3$},
legend style={font=\small,at={(0.025,0.025)},
anchor=south west,thin}
]
\addplot [color=black, mark=o, line width=0.5pt, mark size=2.5pt, mark options={solid, black}] table {data/convergence_p1_ae__error_2_0to40_mh.dat};
\addplot [color=red, mark=square, line width=0.5pt, mark size=2.5pt, mark options={solid, red}] table {data/convergence_p2_ae__error_2_0to40_mh.dat};
\addplot [color=blue, mark=triangle, line width=0.5pt, mark size=3pt, mark options={solid, rotate=270, blue}] table {data/convergence_p3_ae__error_2_0to40_mh.dat};
\logLogSlopeTriangleInv{0.85}{-0.125}{0.34}{-6}{black}{$-6$};
\logLogSlopeTriangleInv{0.85}{-0.12}{0.55}{-4}{black}{$-4$};
\logLogSlopeTriangleInv{0.85}{-0.12}{0.76}{-2}{black}{$-2$};
\end{loglogaxis}
\end{tikzpicture}
\caption{Convergence of the exterior solution for the example in Section \ref{subsec:mexicanHat}. 
The $\text{error} = \max_{i=1,\ldots,N} \vert u^\mathrm{e}(x_i) - u^\mathrm{e}_\ell(x_i) \vert$
is calculated with $N=20$ evaluations points on $\Gamma_\mathrm{e}$. The considered B-Spline spaces have the degrees $p=1,2,3$, respectively. We observe a doubling of the convergence rates.}\label{fig:convergence_Exterior_mh}
\end{figure}
Furthermore, we want to investigate the dependency of the super-convergence on the position of the evaluation point for
a fixed degree $p=3$ of the B-Spline space.
For this, we compare the convergence behavior of the exterior solution on three distinct evaluation paths. 
We denote the paths by $\Gamma_{\mathrm{e},1}$, $\Gamma_{\mathrm{e},2}$, and $\Gamma_{\mathrm{e},3}$, which are the
boundaries of $\left( -1,1 \right)^2$, $\left( -0.35,0.35 \right)^2$, and $\left( -0.26,0.26 \right)^2$, respectively. 
For each evaluation path we choose again $20$ evaluation points to compute the pointwise error \eqref{pointwiseError}. 
The result is visualized in Figure \ref{fig: dependence_pos_mh}, 
where we observe the expected behavior, see Remark \ref{rem:dependence_pos}. 
In particular, super-convergence is readily observed for the solution on $\Gamma_{\mathrm{e},1}$ and $\Gamma_{\mathrm{e},2}$.
We note that for the error in $\Gamma_{\mathrm{e},1}$ we are already at machine precision. 
However, the related constant is larger for the solution on $\Gamma_{\mathrm{e},2}$, 
since the path is closer to the interface boundary $\Gamma=\partial \Omega$ with $\Omega := \left( -0.25,0.25 \right)^2$. 
The same behavior can be observed for the path $\Gamma_{\mathrm{e},3}$, which is even closer to $\Gamma$. 
However, the quality of the computation is also deteriorated in the asymptotic. 
Additionally, we observe saturation effects for higher refinement levels. 
This can be improved by increasing the number of Gaussian quadrature points $N_\mathrm{Gauss}$ on each boundary element, as it is shown in Figure \ref{fig: dependence_gauss_mh}. 
However, this in turn is time consuming. 
With using special extraction techniques, such as the ones developed for $3$-D in \cite{schwab1999extraction}, this undesirable effect
can be reduced. However, a further investigation is beyond the scope of this work.        
\begin{figure}[!tb]
\centering
\begin{tikzpicture}
\begin{loglogaxis}[width=11.5cm,
xlabel={\small $h^{-1}$},
ylabel={\small error},grid=major,
legend entries={$\Gamma_{\mathrm{e},1}$,$\Gamma_{\mathrm{e},2}$,$\Gamma_{\mathrm{e},3}$},
legend style={font=\small,at={(0.025,0.025)},
anchor=south west,thin}
]
\addplot [color=black, mark=o, line width=0.5pt, mark size=2.5pt, mark options={solid, black}] table {data/convergence_p3_ae__error_2_0to40_mh_1.dat};
\addplot [color=red, mark=square, line width=0.5pt, mark size=2.5pt, mark options={solid, red}] table {data/convergence_p3_ae__error_2_0to40_mh.dat};
\addplot [color=blue, mark=triangle, line width=0.5pt, mark size=3pt, mark options={solid, rotate=270, blue}] table {data/convergence_p3_ae__error_2_0to40_mh_026_5.dat};
\logLogSlopeTriangle{0.55}{0.175}{0.15}{-6}{black}{$-6$};
\end{loglogaxis}
\end{tikzpicture}
\caption{Dependence of the super-convergence on the evaluation points. 
The boundary $\Gamma$ of $\left( -0.25,0.25 \right)^2$ is the discretization boundary for the BEM.
The $N=20$ evaluations points to calculate the $\text{error} = \max_{i=1,\ldots,N} \vert u^\mathrm{e}(x_i) - u^\mathrm{e}_\ell(x_i) \vert$
are on $\Gamma_{\mathrm{e},1}$, $\Gamma_{\mathrm{e},2}$, and $\Gamma_{\mathrm{e},3}$, which are the
boundaries of $\left( -1,1 \right)^2$, $\left( -0.35,0.35 \right)^2$, and $\left( -0.26,0.26 \right)^2$, respectively.
We observe the growing constant of the super-convergence constant, which leads to an undesirable saturation for the closest path
$\Gamma_{\mathrm{e},3}$ with respect to $\Gamma$.}\label{fig: dependence_pos_mh}
\end{figure}
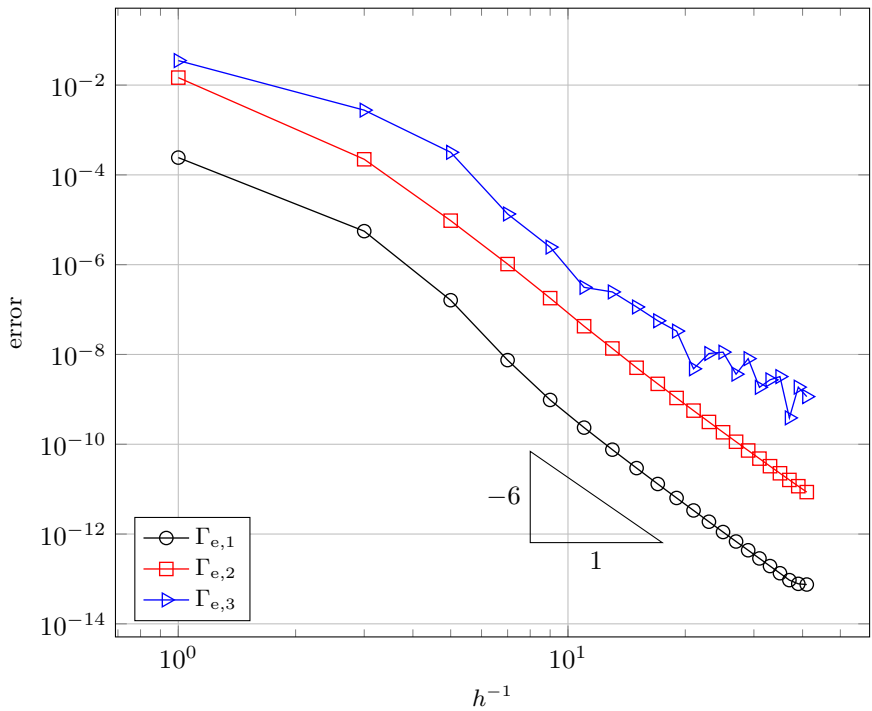%
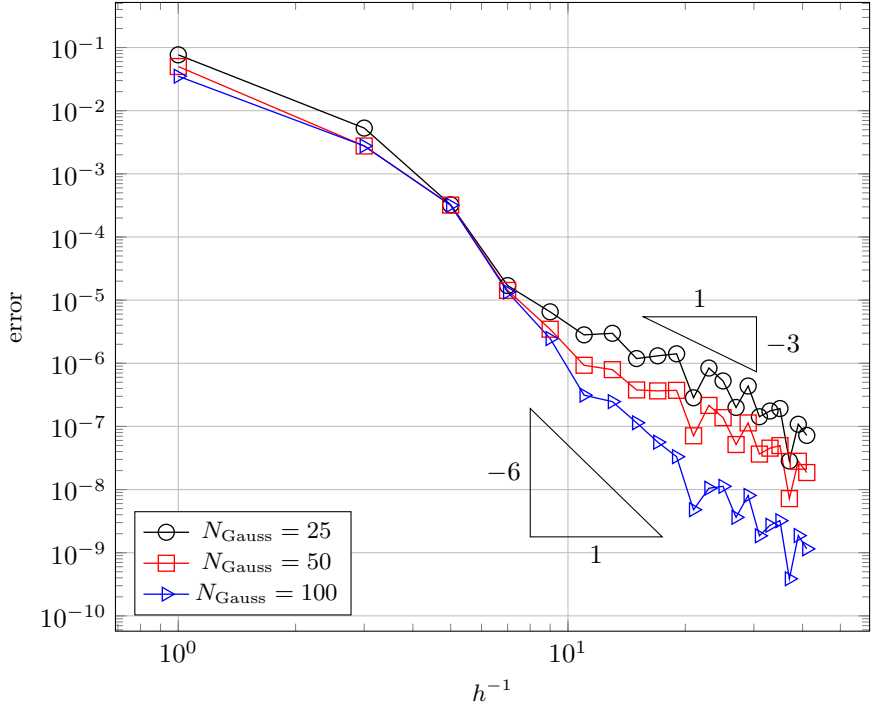
\begin{figure}[!tb]
\centering
\begin{tikzpicture}
\begin{loglogaxis}[width=11.5cm,
xlabel={\small $h^{-1}$},
ylabel={\small error},grid=major,
legend entries={$N_\mathrm{Gauss} = 25$,$N_\mathrm{Gauss} = 50$,$N_\mathrm{Gauss} = 100$},
legend style={font=\small,at={(0.025,0.025)},
anchor=south west,thin}
]
\addplot [color=black, mark=o, line width=0.5pt, mark size=3pt, mark options={solid, rotate=270,black}] table {data/convergence_p3_ae__error_2_0to40.dat};
\addplot [color=red, mark=square, line width=0.5pt, mark size=3pt, mark options={solid, rotate=270, red}] table {data/convergence_p3_ae__error_2_0to40_2.dat};
\addplot [color=blue, mark=triangle, line width=0.5pt, mark size=3pt, mark options={solid, rotate=270, blue}] table {data/convergence_p3_ae__error_2_0to40_mh_026_5.dat};
\logLogSlopeTriangleInv{0.85}{-0.15}{0.5}{-3}{black}{$-3$};
\logLogSlopeTriangle{0.55}{0.175}{0.15}{-6}{black}{$-6$};
\end{loglogaxis}
\end{tikzpicture}
\caption{Dependence of the saturation effect on the number of Gaussian points. 
The boundary $\Gamma$ of $\left( -0.25,0.25 \right)^2$ is the discretization boundary for the BEM.
The evaluations points $N=20$ to calculate $\text{error} = \max_{i=1,\ldots,N} \vert u^\mathrm{e}(x_i) - u^\mathrm{e}_\ell(x_i) \vert$
are on $\Gamma_{\mathrm{e},3}$ of Figure \ref{fig: dependence_pos_mh}, i.e., the
boundary of $\left( -0.26,0.26 \right)^2$. The number of Gaussian points used for the evaluation as well as the assembling of the matrices are $25$, $50$, and $100$, respectively.
We observe an amelioration of the undesirable saturation with increasing the number of Gaussian points. For $N_\mathrm{Gauss} = 100$, the expected super-convergence is restored to some extent.}\label{fig: dependence_gauss_mh}
\end{figure}%
\subsection{Multiple domains}\label{machine}
In this second example, we consider the non-symmetric isogeometric FEM-BEM coupling for the extended boundary value problem \eqref{classProbleme} 
as described in Section \ref{sec:twoDomainProblem}. 
The topology of the model problem and the notation can be adopted from Figure \ref{fig:2domains}. 
However, we consider here a problem domain constructed over circles, see Figure \ref{2ndExple_machine}.  
In particular, if we denote by $B((x_1,x_2);r)$ a circular domain with midpoint $(x_1,x_2)$ and 
radius $r$ we arrive at the following setting:
$\Omega_1=B((0,0);0.39)\backslash B((0,0);0.1)$, $\Omega_2=B((0,0);0.6)\backslash B((0,0);0.4)$, and the thin air gap
$\Omega_\mathrm{b}=B((0,0);0.4)\backslash B((0,0);0.39)$, which describes in fact three rings.
We prescribe the right-hand side $f_i$ in $\Omega_i$ as
\begin{align*}
 f_1(x_1,x_2) = 0 \qquad\text{and}\qquad  f_2(x_1,x_2) =100  \sin(\varphi),
\end{align*} 
where $\varphi$ is the standard angle in a polar coordinate system.
The non-linear material tensor is chosen as
\begin{equation}\label{nonLinearMaterial}
\mathcal{U}_i\nabla u_i : = g(\vert \nabla u_i \vert )\nabla u_i, \, i = 1,2, \quad \text{with } g(t)= \begin{cases}  \frac{1}{150}& \text{ for } t = 0 \\ \frac{1}{100t} \tanh^{-1}(\frac{2t}{3})& \text{ for } 0 < t \leq t_\mathrm{c}:=\frac{3}{2}-\epsilon \\ 1+\beta \exp(-\alpha t) & \text{ for } t > t_\mathrm{c} \end{cases},
\end{equation}  
where we choose $\epsilon > 0$ arbitrarily such that $g(t) < 1$, for all $ 0 < t \leq t_\mathrm{c}$, and $\alpha, \beta$ such that $g(t)$ is continuously differentiable for all $t > 0$\footnote{In this experiment, we choose $\epsilon = 1/100$, $\alpha = g^\prime(t_\mathrm{c})/(1-g(t_\mathrm{c}))$, and $\beta = (g(t_\mathrm{c}) - 1)\exp(\alpha t_\mathrm{c})$.}. In addition, we do not allow jumps, i.e., $u_{0,i}=0$ and $\phi_{0,i}=0$.\\ 
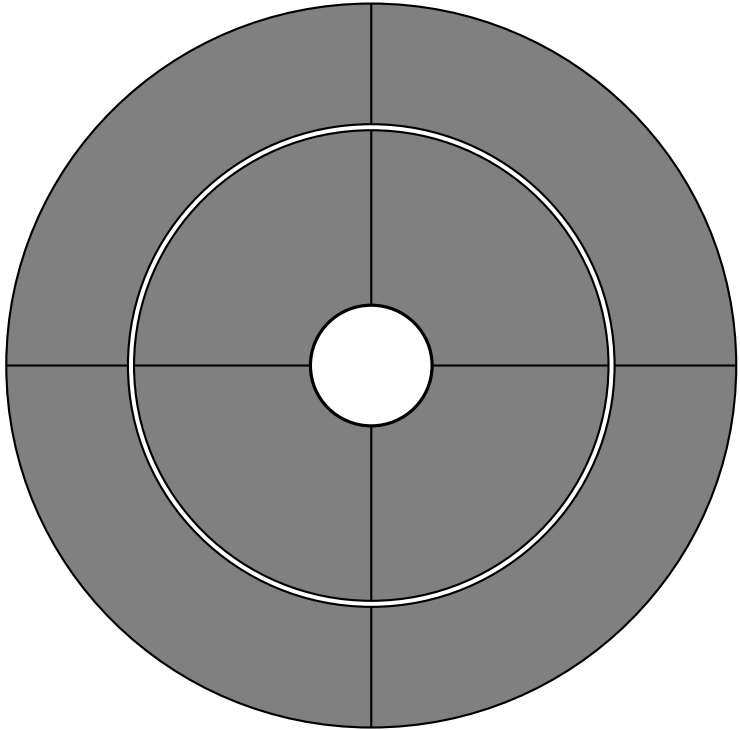
\begin{figure}[!tb]
\begin{center}
\begin{tikzpicture}[scale=0.8]
\draw[fill=gray,thick] (0,0) circle [radius=6] ;
\draw[fill= white, thick] (0,0) circle [radius=4] ;
\draw[fill=gray, thick] (0,0) circle [radius=3.9] ;
\draw[fill= white,very thick] (0,0) circle [radius=1] ;
\draw[-,  thick] (0,1) --(0,3.9);
\draw[-,thick] (0,-1) --(0,-3.9);
\draw[-, thick] (1,0) --(3.9,0);
\draw[-,  thick] (-1,0) --(-3.9,0);
\draw[-,  thick] (0,4) --(0,6);
\draw[-, thick] (0,-4) --(0,-6);
\draw[-, thick] (4,0) --(6,0);
\draw[-, thick] (-4,0) --(-6,0);
\end{tikzpicture}
\end{center}
\caption{Multipatch representation ($4$ patches per domain: the interfaces and boundaries of the patches are highlighted by bold lines) of the example
in Section \ref{machine}.}\label{2ndExple_machine}
\end{figure}~
Following the isogeometric approach, we model both domains separately 
and according to Definition \ref{def:multipatch} as multipatch domains consisting of four patches, see
Figure \ref{2ndExple_machine}. Each patch is represented exactly by a NURBS of degree $p=2$ 
in each parametric direction. Moreover, the Assumption \ref{A5} is obviously satisfied. 
Note that this model configuration with the circular geometry can be interpreted as a $2$D section of a simplified $2$-pole synchronous machine \cite[Section~5.2]{Kurz:diss}. This type of applications motivates also the consideration of non-linear operators. In fact, these devices are mainly made of ferromagnetic materials, which are known to be non-linear.  In particular, by neglecting anisotropies and hysteresis effects, ferromagnetic materials can be modeled by using non-linear operators of the same type as the ones we considered in Lemma \ref{dataDep_interface} and Lemma \ref{dataDep_interfaceBVP}, and for this example in \eqref{nonLinearMaterial}. For more details about this topic, see \cite{pechstein:thesis} and \cite{roemer}, for instance. Furthermore, we refer to \cite{zeger} for electrical 
engineering simulations of electric machines. \\
In this experiment, the arising non-linear problem is solved by using a standard Picard iteration method. For the stopping criterion, we consider a relative residual error of $10^{-10}$. 
In our simulation below we need an average of $35$ Picard iterations to fulfill the criterion. \\
The solutions $u_1$ and $u_2$ in the interior domains $\Omega_1$ and $\Omega_2$, respectively, 
 are visualized in Figure \ref{fig:sol_exple2_machine}. In the context of electric machines, $u_i,\,i=1,2,$ can be 
 interpreted as the third component of the magnetic vector potential. Note that the equipotential lines, i.e., the continuous black lines in Figure \ref{fig:sol_exple2_machine} 
are the magnetic field lines. The interaction of the magnetic fields stemming 
 from the rotor and the stator in the air gap may induce a mechanical torque. This leads the rotor, i.e., the interior ring to move in order to 
 reduce the (spatial) phase shift between both magnetic fields. In particular, the computation of torques and forces involves the computation of the magnetic flux density, which in turn requires the evaluation of the solution and its derivatives in the gap domain, c.f., \cite{kurzMotor}, for instance. Therefore, the super-convergence behavior in the air gap of the machine is of particular interest.\\
 In a post-processing step, we compute the magnetic reluctivity, which is defined as the reciprocal of the magnetic permeability. Formally, it is given by the function $g(\vert \nabla u_i \vert ), \, i=1,2,$ in \eqref{nonLinearMaterial}, which we evaluate by using the solutions $u_1$ and $u_2$. Since we are considering non-linear materials, the reluctivity is not constant across the electric machine. This is depicted in Figure \ref{fig:sol_exple2_machine_saturationEffect}. Note that the thick black lines are the same as those in Figure \ref{fig:sol_exple2_machine}, i.e., they represent the equipotential lines of the solution $u_1$ and $u_2$.  \\   
\begin{figure}[!tb]
\begin{center}
\includegraphics[width=11.5cm,trim=0.5cm 2.5cm 0cm 1.5cm,clip]{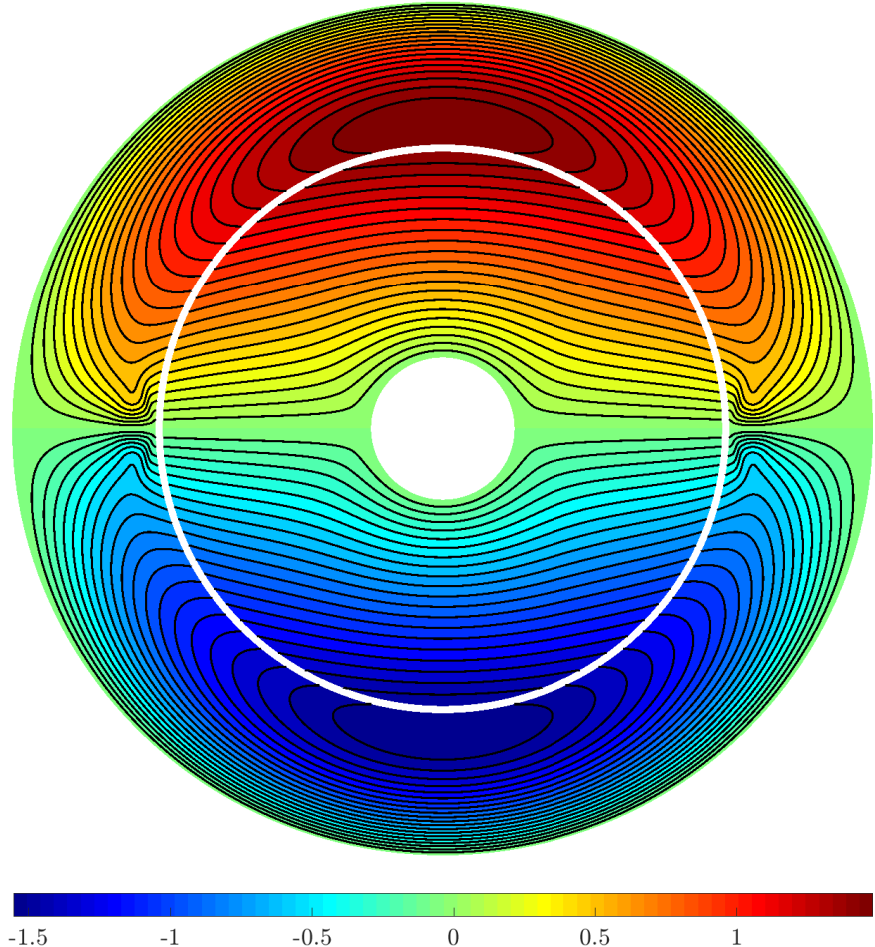}
\end{center}
\caption{The solution in the interior domains $\Omega_1$ and $\Omega_2$ for the electric machine of the example
in Section \ref{machine}. The equipotential lines
are the magnetic field lines.} \label{fig:sol_exple2_machine}
\end{figure}
\begin{figure}[!tb]
\begin{center}
\includegraphics[width=11.5cm,trim=0.5cm 2.5cm 0cm 1.5cm,clip]{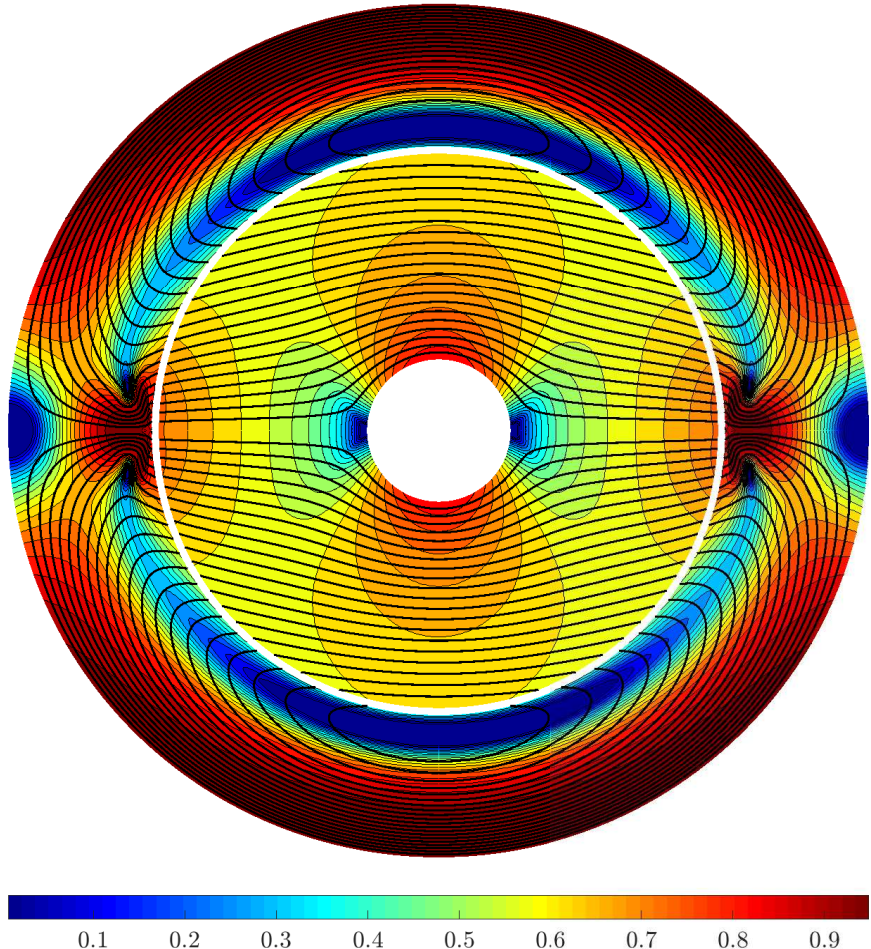}
\end{center}
\caption{Saturation effects caused by the non-linear material tensors. The color-bar and the thin lines 
represent the levels of the magnetic reluctivity, which is given by $g(\nabla u_i), $ for $i=1,2$, and
 evaluated using the derivatives of the solutions $u_1$ and $u_2$. The thick equipotential lines show the flow direction of the magnetic field.} \label{fig:sol_exple2_machine_saturationEffect}
\end{figure}
To verify Remark \ref{rem:dependence_pos} numerically, 
we evaluate the solution in the BEM domain $\Omega_\mathrm{b}$
on the evaluation path given as the parametrized circle $\partial B((0,0);0.395)$. Note that in this example, the BEM is applied in an interior domain $\Omega_\mathrm{b}$.
Hence, we use for the evaluation the representation formula \eqref{repFormula_intro} with $\kappa=0$ and the  
complete Cauchy data on $\Gamma_1$ and $\Gamma_2$, which are available after solving Problem \ref{discreteextended} with
the jump condition \eqref{dirJumpe}. 
An analytical solution for our model problem is not known. Hence, to verify the convergence order, we follow a standard procedure:
The mesh of the current solution is successively refined three times and we calculate the corresponding discrete solutions. 
We apply the Aitkin's $\Delta^2$-extrapolation to this sequence of discrete solutions and this extrapolated value is
the reference solution $u^\mathrm{e}(x_i)$ for the 
$\text{error} = \max_{i=1,\ldots,N} \vert u^\mathrm{e}(x_i) - u^\mathrm{e}_\ell(x_i) \vert$ calculated from $N=20$ evaluations points. 
This error is visualized in Figure \ref{fig:convergence_airGapMultipleDomain_machine_thin} 
for ansatz spaces of degree $p=2$ and $p=3$, where we observe an amelioration of the convergence rates. Note that this amelioration depends on the quality of the numerical integration, as shown in Figure \ref{fig: dependence_gauss_mh} for the example of Section \ref{subsec:mexicanHat}. For this example, noticeable amelioration of the convergence rates were only observable for a high number of Gaussian quadrature points. In this case, $N_\mathrm{Gauss} = 400$ points were considered for the assembling of the BEM matrices, which is very time consuming. The dominance of the quadrature error for this type of evaluations can however be tackled, as mentioned in the previous section, by using special extraction techniques. Moreover, efficient assembly of the BEM matrices based on B-spline tailored quadrature rules, as given in \cite{bSplineBasedQuad}, together with suitable compression methods, see e.g., \cite{doelz2019isogeometric}, would accelerate the computation considerably. However, this investigation is beyond the scope of this work.     
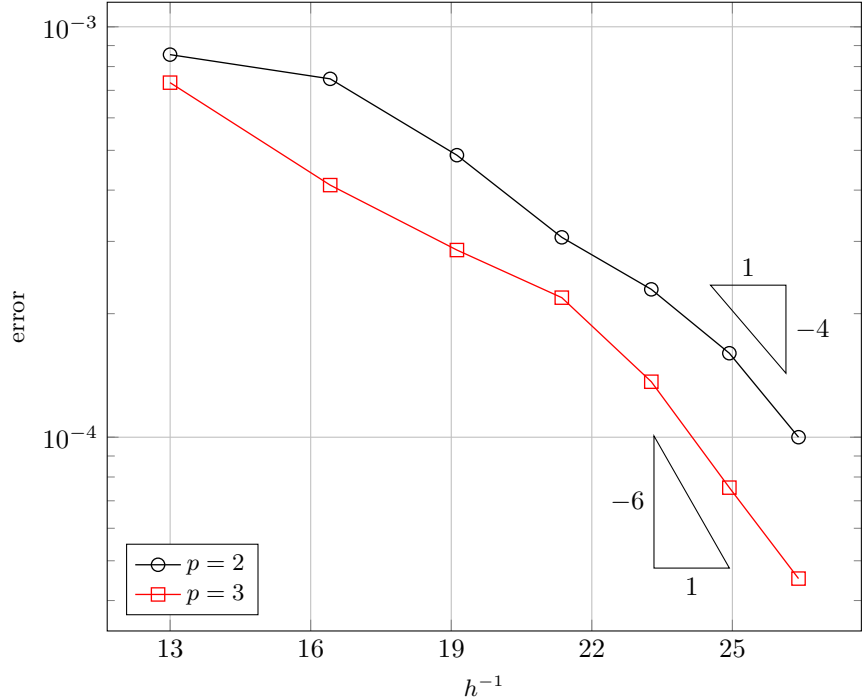
\begin{figure}[!tb]
\centering
\begin{tikzpicture}
\begin{loglogaxis}[width=11.5cm,
xlabel={\small $h^{-1}$},
xticklabels={$10$, $13$, $16$, $19$, $22$, $25$, $28$},
ylabel={\small error},grid=major,
legend entries={$p = 2$, $p=3$},
legend style={font=\small,at={(0.025,0.025)},
anchor=south west,thin}
]
\addplot [color=black, mark=o, line width=0.5pt, mark size=2.5pt, mark options={solid, black}] table {data/convergence_p2_ae__error_2_0to14_machine_realNonlinearityBigGap.dat};
\addplot [color=red, mark=square, line width=0.5pt, mark size=2.5pt, mark options={solid, red}] table {data/convergence_p3_ae__error_2_0to14_machine_realNonlinearityBigGap.dat};
\logLogSlopeTriangle{0.725}{0.1}{0.1}{-6}{black}{$-6$};
\logLogSlopeTriangleInv{0.9}{-0.1}{0.55}{-4}{black}{$-4$};
\end{loglogaxis}
\end{tikzpicture}
\caption{Convergence of the solution on the evaluation path $\Gamma_\mathrm{e}= \partial B((0,0);0.395)$ (circle) 
in the air gap $\Omega_\mathrm{b}$ for
the example in Section \ref{machine}. 
The $\text{error} = \max_{i=1,\ldots,N} \vert u^\mathrm{e}(x_i) - u^\mathrm{e}_\ell(x_i) \vert$
is calculated with $N=20$ evaluations points on $\Gamma_\mathrm{e}$. As a replacement for the unknown analytical solution
we use an Aitken $\Delta^2$ extrapolation of a sequence of successively refined discrete solutions.} \label{fig:convergence_airGapMultipleDomain_machine_thin}
\end{figure}%
\section{Conclusions}
The non-symmetric FEM-BEM coupling in the isogeometric context for simulating
practical problems with complex geometries turns out to be a promising alternative to classical approaches. 
A transformation to an integral formulation allows a problem in a domain to be reduced to its boundary, where the BEM can be
applied. For exterior problems there is no need to truncate the unbounded domain, and for simulating thin gaps 
there is no need for a complicated remeshing. In both cases numerical errors can be avoided. 
Thanks to the definition of B-Splines, $h$- and $p$-refinements are applied in a straightforward manner. 
Furthermore, multiple domain modeling can be done independently. This is particularly advantageous
if we consider moving or deforming geometries. A classical transmission and a multiple domain problem with parts
of non-linear material are considered. Obviously, FEM is applied to the non-linear areas, whereas BEM is exclusively used for
the linear problem.
For both model problems, well-posedness for the continuous and discrete problem, and quasi-optimality and convergence rates 
for the numerical approximation
are mathematically analyzed in the isogeometric framework. 
Furthermore, we show an improvement of the convergence behavior, 
if we consider the error in functionals of the solution. 
This is motivated by a practice-oriented application 
such as electric machines. Here the computation of torques are a central task and involve the evaluation of some 
derivatives of the solution in the BEM domain. We observe for both model applications this super-convergence, which confirms 
the theory. Future extensions of the method may include the consideration of 
parabolic-elliptic problems and a rigorous analysis of the coupling for $\curl \curl$-type equations in $3$D.  %

\bibliography{igafembem_preprint}

\newcommand{\etalchar}[1]{$^{#1}$}
\begin{thebibliography}{BBadVC{\etalchar{+}}06}

\bibitem[ACD{\etalchar{+}}18]{bSplineBasedQuad}
A.~Aimi, F.~Calabro, M.~Diligenti, M.~L. Sampoli, G.~Sangalli, and A.~Sestini.
\newblock Efficient assembly based on {B}-spline tailored quadrature rules for
  the {IgA-SGBEM}.
\newblock {\em Computer Methods in Applied Mechanics and Engineering},
  331:327--342, 2018.

\bibitem[AFF{\etalchar{+}}13]{aurada2013coupling}
M.~Aurada, M.~Feischl, T.~F{\"u}hrer, M.~Karkulik, J.~M. Melenk, and
  D.~Praetorius.
\newblock Classical {FEM}-{BEM} coupling methods: nonlinearities,
  well-posedness, and adaptivity.
\newblock {\em Computational Mechanics}, 51(4):399--419, 2013.

\bibitem[BadVBSV14]{da2014mathematical}
L.~Beir\~ao~da Veiga, A.~Buffa, G.~Sangalli, and R.~V{\'a}zquez.
\newblock Mathematical analysis of variational isogeometric methods.
\newblock {\em Acta Numerica}, 23:157--287, 2014.

\bibitem[Ban15]{bantle}
A.~Bantle.
\newblock {\em On high-order {NURBS}-based boundary element methods in two
  dimensions-numerical integration and implementation}.
\newblock PhD thesis, Fakult{\"a}t f{\"u}r Mathematik und
  Wirtschaftswissenschaften, Universit{\"a}t Ulm, 2015.

\bibitem[BBadVC{\etalchar{+}}06]{Bazilevs2006}
Y.~Bazilevs, L.~Beir\~ao~da Veiga, J.~A. Cottrell, T.~J.~R. Hughes, and
  G.~Sangalli.
\newblock Isogeometric analysis: approximation, stability and error estimates
  for h-refined meshes.
\newblock {\em Mathematical Models and Methods in Applied Sciences},
  16(07):1031--1090, 2006.

\bibitem[BCSDG17]{zeger}
Z.~Bontinck, J.~Corno, S.~Schöps, and H.~De~Gersem.
\newblock Isogeometric {A}nalysis and {H}armonic {S}tator-{R}otor {C}oupling
  for {S}imulating {E}lectric {M}achines.
\newblock {\em Computer Methods in Applied Mechanics and Engineering}, 334, 09
  2017.

\bibitem[BDK{\etalchar{+}}20]{buffa2020multipatch}
A.~Buffa, J.~D{\"o}lz, S.~Kurz, S.~Sch{\"o}ps, R.~V{\'a}zquez, and F.~Wolf.
\newblock Multipatch approximation of the de {R}ham sequence and its traces in
  isogeometric analysis.
\newblock {\em Numerische Mathematik}, 144(1):201--236, 2020.

\bibitem[BM83]{bielak1983exterior}
J.~Bielak and R.~C. Mac{C}amy.
\newblock An exterior interface problem in two-dimensional elastodynamics.
\newblock {\em Quarterly of Applied Mathematics}, 41(1):143--159, 1983.

\bibitem[CdFDGS16]{corno2016isogeometric}
J.~Corno, C.~de~Falco, H.~De~Gersem, and S.~Sch{\"o}ps.
\newblock Isogeometric simulation of {L}orentz detuning in superconducting
  accelerator cavities.
\newblock {\em Computer Physics Communications}, 201:1--7, 2016.

\bibitem[CHB09]{cottrell2009isogeometric}
J.~A. Cottrell, T.~JR Hughes, and Y.~Bazilevs.
\newblock {\em Isogeometric analysis: toward integration of {CAD} and {FEA}}.
\newblock John Wiley \& Sons, 2009.

\bibitem[Cos88a]{costabel}
M.~Costabel.
\newblock Boundary integral operators on {L}ipschitz domains: elementary
  results.
\newblock {\em SIAM Journal on Mathematical Analysis}, 19(3):613--626, 1988.

\bibitem[Cos88b]{costabel1988symmetric}
M.~Costabel.
\newblock A symmetric method for the coupling of finite elements and boundary
  elements.
\newblock {\em The Mathematics of Finite Elements and Applications, VI
  (Uxbridge, 1987)}, pages 281--288, 1988.

\bibitem[dFRV11]{geopdes}
C.~de~Falco, A.~Reali, and R.~V{\'a}zquez.
\newblock Geo{PDE}s: a research tool for isogeometric analysis of {PDEs}.
\newblock {\em Advances in Engineering Software}, 42(12):1020--1034, 2011.

\bibitem[DKSW19]{doelz2019isogeometric}
J.~D{\"o}lz, S.~Kurz, S.~Sch{\"o}ps, and F.~Wolf.
\newblock Isogeometric boundary elements in electromagnetism: Rigorous
  analysis, fast methods, and examples.
\newblock {\em SIAM Journal on Scientific Computing}, 41(5):B983--B1010, 2019.

\bibitem[EES18]{Erath:2018-1}
H.~Egger, C.~Erath, and R.~Schorr.
\newblock On the nonsymmetric coupling method for parabolic-elliptic interface
  problems.
\newblock {\em SIAM J. Numer. Anal.}, 56(6):3510--3533, 2018.

\bibitem[EOS17]{erath2017}
C.~Erath, G.~Of, and F.~Sayas.
\newblock A non-symmetric coupling of the finite volume method and the boundary
  element method.
\newblock {\em Numerische Mathematik}, 135(3):895--922, 2017.

\bibitem[Era12]{Erath:2012-1}
C.~Erath.
\newblock Coupling of the finite volume element method and the boundary element
  method: an a priori convergence result.
\newblock {\em SIAM J. Numer. Anal.}, 50(2):574--594, 2012.

\bibitem[FGHP16]{FEISCHL2016141}
M.~Feischl, G.~Gantner, A.~Haberl, and D.~Praetorius.
\newblock Adaptive {2D} {IGA} boundary element methods.
\newblock {\em Engineering Analysis with Boundary Elements}, 62:141 -- 153,
  2016.

\bibitem[FGP15]{FEISCHL2015362}
M.~Feischl, G.~Gantner, and D.~Praetorius.
\newblock Reliable and efficient a posteriori error estimation for adaptive
  {IGA} boundary element methods for weakly-singular integral equations.
\newblock {\em Computer Methods in Applied Mechanics and Engineering}, 290:362
  -- 386, 2015.

\bibitem[FGPS19]{FUHRER2019571}
T.~Führer, G.~Gantner, D.~Praetorius, and S.~Schimanko.
\newblock Optimal additive schwarz preconditioning for adaptive {2D} {IGA}
  boundary element methods.
\newblock {\em Computer Methods in Applied Mechanics and Engineering}, 351:571
  -- 598, 2019.

\bibitem[Gan14]{GantnerGregor2014AiB}
G.~Gantner.
\newblock Adaptive isogeometrische {BEM}.
\newblock Master's thesis, Vienna University of Technology, 2014.

\bibitem[Gan17]{GantnerGregor2017Oafs}
G.~Gantner.
\newblock {\em Optimal adaptivity for splines in finite and boundary element
  methods}.
\newblock PhD thesis, Vienna University of Technology, 2017.

\bibitem[HCB05]{HUGHES20054135}
T.J.R. Hughes, J.A. Cottrell, and Y.~Bazilevs.
\newblock Isogeometric analysis: {CAD}, finite elements, {NURBS}, exact
  geometry and mesh refinement.
\newblock {\em Computer Methods in Applied Mechanics and Engineering},
  194(39):4135 -- 4195, 2005.

\bibitem[HR09]{Habrecht2010}
H.~Harbrecht and M.~Randrianarivony.
\newblock From computer aided design to wavelet {BEM}.
\newblock {\em Computer Methods in Applied Mechanics and Engineering}, 13(2):69
  -- 82, 2009.

\bibitem[JN80]{johnson1980coupling}
C.~Johnson and J.~C. N{\'e}d{\'e}lec.
\newblock On the coupling of boundary integral and finite element methods.
\newblock {\em Mathematics of computation}, 35:1063--1079, 1980.

\bibitem[KFK{\etalchar{+}}97]{kurzMotor}
S.~Kurz, J.~Fetzer, T.~Kube, G.~Lehner, and W.~M. Rucker.
\newblock {BEM}-{FEM} coupling in electromechanics: A 2-{D} watch stepping
  motor driven by a thin wire coil.
\newblock {\em Applied Computational Electromagnetics Society Journal},
  12:135--139, 1997.

\bibitem[Kur98]{Kurz:diss}
S.~Kurz.
\newblock {\em Die numerische Behandlung elektromechanischer Systeme mit Hilfe
  der Kopplung der Methode der finiten Elemente und der Randelementmethode}.
\newblock VDI-Verlag, 1998.

\bibitem[McL00]{mclean}
W.~McLean.
\newblock {\em Strongly elliptic systems and boundary integral equations}.
\newblock Cambridge university press, 2000.

\bibitem[MS87]{MacCamy:1987}
R.~C. Mac{C}amy and M.~Suri.
\newblock A time-dependent interface problem for two-dimensional eddy currents.
\newblock {\em Quart. Appl. Math.}, 44:675--690, 1987.

\bibitem[MZBF15]{MARUSSIG2015458}
B.~Marussig, J.~Zechner, G.~Beer, and T.-P. Fries.
\newblock Fast isogeometric boundary element method based on independent field
  approximation.
\newblock {\em Computer Methods in Applied Mechanics and Engineering}, 284:458
  -- 488, 2015.
\newblock Isogeometric Analysis Special Issue.

\bibitem[OS13]{of2013one}
G.~Of and O.~Steinbach.
\newblock Is the one-equation coupling of finite and boundary element methods
  always stable?
\newblock {\em ZAMM-Journal of Applied Mathematics and Mechanics/Zeitschrift
  f{\"u}r Angewandte Mathematik und Mechanik}, 93(6-7):476--484, 2013.

\bibitem[OS14]{of2014ellipticity}
G.~Of and O.~Steinbach.
\newblock On the ellipticity of coupled finite element and one-equation
  boundary element methods for boundary value problems.
\newblock {\em Numerische Mathematik}, 127(3):567--593, 2014.

\bibitem[Pec04]{pechstein:thesis}
C.~Pechstein.
\newblock {\em Multigrid-Newton-methods for nonlinear magnetostatic problems}.
\newblock Master's thesis, Johannes Kepler University of Linz, Institute of
  Computational Mathematics, Linz, 2004.

\bibitem[PGK{\etalchar{+}}09]{Costas2009}
C.~Politis, A.~I. Ginnis, P.~D. Kaklis, K.~Belibassakis, and C.~Feurer.
\newblock An isogeometric {BEM} for exterior potential-flow problems in the
  plane.
\newblock In {\em 2009 SIAM/ACM Joint Conference on Geometric and Physical
  Modeling}, SPM ’09, page 349–354, New York, NY, USA, 2009. Association
  for Computing Machinery.

\bibitem[PT12]{piegl2012nurbs}
L.~Piegl and W.~Tiller.
\newblock {\em The NURBS book}.
\newblock Springer Science \& Business Media, 2012.

\bibitem[R{\"o}m15]{roemer}
U.~R{\"o}mer.
\newblock {\em Numerical approximation of the magnetoquasistatic model with
  uncertainties and its application to magnet design}.
\newblock PhD thesis, Institut f{\"u}r Theorie Elektromagnetischer Felder,
  Technische Universit{\"a}t Darmstadt, 2015.

\bibitem[Say09]{sayas2009validity}
F.~Sayas.
\newblock The validity of {J}ohnson--{N}{\'e}d{\'e}lec's {BEM}--{FEM} coupling
  on polygonal interfaces.
\newblock {\em SIAM Journal on Numerical Analysis}, 47(5):3451--3463, 2009.

\bibitem[SS10]{sauter}
S.~A. Sauter and C.~Schwab.
\newblock Boundary {E}lement {M}ethods.
\newblock In {\em Boundary Element Methods}, pages 183--287. Springer, 2010.

\bibitem[SSE{\etalchar{+}}13]{SCOTT2013197}
M.A. Scott, R.N. Simpson, J.A. Evans, S.~Lipton, S.P.A. Bordas, T.J.R. Hughes,
  and T.W. Sederberg.
\newblock Isogeometric boundary element analysis using unstructured
  {T}-splines.
\newblock {\em Computer Methods in Applied Mechanics and Engineering}, 254:197
  -- 221, 2013.

\bibitem[Ste07]{steinbach}
O.~Steinbach.
\newblock {\em Numerical approximation methods for elliptic boundary value
  problems: finite and boundary elements}.
\newblock Springer Science \& Business Media, 2007.

\bibitem[Ste11]{steinbach2011note}
O.~Steinbach.
\newblock A note on the stable one-equation coupling of finite and boundary
  elements.
\newblock {\em SIAM journal on numerical analysis}, 49(4):1521--1531, 2011.

\bibitem[SW99]{schwab1999extraction}
C.~Schwab and W.~Wendland.
\newblock On the extraction technique in boundary integral equations.
\newblock {\em Mathematics of computation}, 68(225):91--122, 1999.

\bibitem[Zei86]{zeidler}
E.~Zeidler.
\newblock {\em Nonlinear Functional Analysis and its Applications II/B:
  Nonlinear Monotone Operators}.
\newblock Springer, 1986.

\end{thebibliography}
\bibliographystyle{alpha}
\end{document}